\pdfoutput=1
\RequirePackage{ifpdf}
\ifpdf 
\documentclass[pdftex]{sigma}
\else
\documentclass{sigma}
\fi

\numberwithin{equation}{section}

\newtheorem{Theorem}{Theorem}[section]
\newtheorem{Corollary}[Theorem]{Corollary}
\newtheorem{Lemma}[Theorem]{Lemma}
\newtheorem{Proposition}[Theorem]{Proposition}
 { \theoremstyle{definition}
\newtheorem{Definition}[Theorem]{Definition}
\newtheorem{Remark}[Theorem]{Remark}
\newtheorem{Example}[Theorem]{Example}
 }

\usepackage{mathrsfs}
\usepackage[all]{xy}
\usepackage{enumerate}

\usepackage[Symbol]{upgreek}

\usepackage{tikz}
\usetikzlibrary{arrows,shapes,backgrounds}
\usepgflibrary{arrows}
\usetikzlibrary{topaths}
\usetikzlibrary{er}

\newcommand{\tensor}[1]{\otimes_{\scriptscriptstyle{#1}}}

\newcommand{\Sf}[1]{\mathsf{#1}}
\newcommand{\fk}[1]{\mathfrak{#1}}

\newcommand{\lmod}[1]{{}_{#1}\Sf{Mod}}

\renewcommand{\hom}[3]{\mathrm{Hom}_{\Sscript{#1}}\left(#2, #3\right)}

\newcommand{\td}[1]{\widetilde{#1}}

\newcommand{\bd}[1]{\boldsymbol{#1}}

\newcommand{\bara}[1]{\overline{#1}}

\newcommand{\Rep}[1]{\mathbf{Rep}_{\Sscript{\Bbbk}}(#1)}

\newcommand{\Vect}{{\rm Vect}_{\Sscript{\Bbbk}}}
\newcommand{\vect}{{\rm vect}_{\Sscript{\Bbbk}}}

\newcommand{\tF}{\texttt{f}}

\newcommand{\phil}{\upphi_{\Sscript{*}}}

\newcommand{\Go}{\cG_{\Sscript{0}}}
\newcommand{\Ga}{\cG_{\Sscript{1}}}
\newcommand{\Ho}{\cH_{\Sscript{0}}}
\newcommand{\Ha}{\cH_{\Sscript{1}}}
\newcommand{\Ko}{\cK_{\Sscript{0}}}
\newcommand{\Ka}{\cK_{\Sscript{1}}}
\newcommand{\No}{\cN_{\Sscript{0}}}
\newcommand{\Na}{\cN_{\Sscript{1}}}

\newcommand{\Co}{\mathcal{C}_{\Sscript{0}}}
\newcommand{\Ca}{\mathcal{C}_{\Sscript{1}}}
\newcommand{\Do}{\cD_{\Sscript{0}}}
\newcommand{\Da}{\cD_{\Sscript{1}}}

\newcommand{\Fo}{F_{\Sscript{0}}}
\newcommand{\Fa}{F_{\Sscript{1}}}

\newcommand{\Vx}{\cV_{\Sscript{x}}}

\newcommand{\Ux}{\cU_{\Sscript{x}}}

\newcommand{\Wu}{\cW_{\Sscript{u}}}

\newcommand{\Vg}{\cV^{\Sscript{g}}}

\newcommand{\Umg}{\cU^{\Sscript{g}^{-1}}}

\newcommand{\Wh}{\cW^{\Sscript{h}}}
\newcommand{\We}{\cW^{\Sscript{e}}}

\newcommand{\Ker}[1]{{\rm Ker}(#1)}

\newcommand{\Starl}[1]{{{\rm Star}}^{\Sscript{l}}(#1)}

\newcommand{\Ad}[1]{{\bd{ad}}_{\Sscript{#1}}}
\newcommand{\Orbit}[2]{{{\rm Orb}}_{\Sscript{\mathcal{#2}}}^{\Sscript{l}}\big(
#1 \big)}

\newcommand{\mycircled}[2][none]{\tikz[baseline=(a.base)]\node[draw,circle,inner
sep=0.5pt, outer sep=0pt,fill=#1](a){\ensuremath #2\strut};
 }

\newcommand{\oO}{\mathscr{O}}

\newcommand{\uU}{\mathscr{U}}
\newcommand{\vV}{\mathscr{V}}

\newcommand{\cD}{{\mathcal D}}
\newcommand{\cE}{{\mathcal E}}

\newcommand{\cG}{{\mathcal G}}
\newcommand{\cH}{{\mathcal H}}

\newcommand{\cK}{{\mathcal K}}

\newcommand{\cN}{{\mathcal N}}

\newcommand{\cR}{{\mathcal R}}

\newcommand{\cU}{{\mathcal U}}
\newcommand{\cV}{{\mathcal V}}
\newcommand{\cW}{{\mathcal W}}

\newcommand{\cY}{{\mathcal Y}}

\newcommand{\Ni}{\mathcal{N}^{\Sscript{(i)}}}
\newcommand{\Hi}{\mathcal{H}^{\Sscript{(i)}}}

\newcommand{\Gx}{\mathcal{G}^{\Sscript{x}}}

\newcommand{\phix}{\upphi^{\Sscript{x}}}

\newcommand{\phixp}{\upphi^{\Sscript{x'}}}

\newcommand{\Sets}{\mathsf{Sets}}
\newcommand{\Gsets}{\cG\text{-}{\rm Sets}}

\newcommand{\Sscript}[1]{\scriptscriptstyle{#1}}
\newcommand{\due}[3]{{}_{{#2 }} {#1}_{{ #3}} }

\newcommand{\prolimit}[2]{\varprojlim_{#1}\left({#2}\right)}
\newcommand{\injlimit}[2]{\varinjlim_{#1}\left({#2}\right)}
\newcommand{\upg}{\uU^{\Sscript{\upphi}}(\cG)}
\newcommand{\pug}{{}^{\Sscript{\upphi}}\uU(\cG)}
\newcommand{\pls}{{}^{*}\upphi}
\newcommand{\phio}{\upphi_{\Sscript{0}}}
\newcommand{\phia}{\upphi_{\Sscript{1}}}
\newcommand{\Vgx}{\cV^{\Sscript{\cG}}_{\Sscript{x}}}
\newcommand{\Vgsg}{\cV^{\Sscript{\cG}}_{\Sscript{s(g)}}}
\newcommand{\Vgtg}{\cV^{\Sscript{\cG}}_{\Sscript{t(g)}}}

\newcommand*{\QEDA}{\hfill\ensuremath{\blacksquare}}%

\begin{document}

\allowdisplaybreaks

\newcommand{\arXivNumber}{1806.09327}

\renewcommand{\PaperNumber}{019}

\FirstPageHeading

\ShortArticleName{Linear Representations and Frobenius Morphisms of Groupoids}

\ArticleName{Linear Representations and Frobenius Morphisms\\ of Groupoids}

\Author{Juan Jes\'us BARBAR\'AN S\'ANCHEZ~$^\dag$ and Laiachi EL KAOUTIT~$^\ddag$}

\AuthorNameForHeading{J.J.~Barbar\'an S\'anchez and L.~El Kaoutit}

\Address{$^\dag$~Universidad de Granada, Departamento de \'{A}lgebra, Facultad de Educaci\'{o}n,\\
\hphantom{$^\dag$}~Econon\'ia y Tecnolog\'ia de Ceuta, Cortadura del Valle, s/n.~E-51001 Ceuta, Spain}
\EmailD{\href{mailto:barbaran@ugr.es}{barbaran@ugr.es}}
\URLaddressD{\url{http://www.ugr.es/~barbaran/}}

\Address{$^\ddag$~Universidad de Granada, Departamento de \'{A}lgebra and IEMath-Granada, \\
\hphantom{$^\dag$}~Facultad de Educaci\'{o}n, Econon\'ia y Tecnolog\'ia de Ceuta, Cortadura del Valle,\\
\hphantom{$^\dag$}~s/n.~E-51001 Ceuta, Spain}
\EmailD{\href{mailto:kaoutit@ugr.es}{kaoutit@ugr.es}}
\URLaddressD{\url{http://www.ugr.es/~kaoutit/}}

\ArticleDates{Received June 26, 2018, in final form February 22, 2019; Published online March 12, 2019}

\Abstract{Given a morphism of (small) groupoids with injective object map, we provide sufficient and necessary conditions under which the
induction and co-induction functors between the categories of linear representations are naturally isomorphic. A morphism with
this property is termed a \emph{Frobenius morphism of groupoids}. As a consequence, an extension by a subgroupoid is Frobenius if and
only if each fibre of the (left or right) pull-back biset has finitely many orbits. Our results extend and clarify the classical Frobenius reciprocity formulae in the theory of finite groups, and characterize Frobenius extension of algebras with enough orthogonal idempotents.}


\Classification{18B40, 20L05, 20L99; 18D10,16D90, 18D35}

\newcommand{\rJoin}{\rtimes}
\newcommand{\lJoin}{\ltimes}

\vspace{-2mm}

\section{Introduction}

In this section, we first explain the motivations behind this research and we give a general overview of the theory developed here. Secondly, we introduce the notations and conventions that are needed in order to give a detailed description of the main results obtained in this paper and to make this introduction self-contained as much as possible.

\subsection{Motivation and overview}\label{ssec:1} Either as abstract objects or as geometrical ones, groupoids appear in different branches of mathematics and mathematical physics: see for instance the brief surveys \cite{Brown:1987, Landsman:2005,Weinstein:1996}. It seems that the most common motivation for studying groupoids has its roots in the concept of symmetry and in the knowledge of its formalism. Apparently, groupoids do not only allow to consider symmetries coming from transformations of the object (i.e., algebraic and/or geometric automorphisms), but they also allow to deal with symmetries among the parts of the object.

As it was claimed in \cite{Guay/Hepburn:2009}, to find a proper generalization of the formal definition of symmetry, one doesn't need to consider the class of all groupoids: for that, the \emph{equivalence relation} and \emph{action groupoids} serve as an intermediate step. From an abstract point of view, \emph{equivalence relation groupoids} are too restrictive as they do not admit any non trivial isotropy group. In other words, there is no internal symmetry to be considered when these groupoids are employed. Concerning \emph{action groupoids} and their linear representations, where internal symmetries appear, it is noteworthy to mention that they have been manifested implicitly in several physical situations a long time ago. In terms of homogeneous vector bundles, the study of molecular vibration is, for instance, a situation where linear representations of action groupoids are exemplified (see \cite[Section~3.2, p.~97]{Sternberg:1994} for more details in the specific case of the space of motions of carbon tetrachloride and~\cite{Pysiak:2011} for others examples).

Let us explain with some details how this exemplification appears in a general situation. Assume a group $G$ is given together with a right $G$-set $M$ (see \cite{Bouc:2010} for the precise definition) and consider the associated action groupoid $(M\times G, M)$ as in Example \ref{exam:action} below. Then, any abstract homogeneous vector bundle $\pi\colon E \to M$ over $M$\footnote{In the aforementioned physical situation, $M$ is the finite set of four chlorine atoms at the vertices of a regular tetrahedron including the carbon atom located in the centre and each fibre of $E$ is the three-dimensional vector space, which describes the displacement of the atom from its equilibrium position. The acting group is $S_{\Sscript{4}}$, the symmetric group of four elements. The global sections of $E$ are functions that parametrize the displacements of the molecule in its whole shape.} leads to a linear representation on the action groupoid $(M\times G, M)$ given by the functor $\{x \mapsto E_{\Sscript{x}}\}_{\Sscript{x \in M}}$, where the vector spaces $E_x$'s are the fibres of~$(E,\pi)$. This bundle also leads to a morphism $(\pi \times G, \pi)\colon (E \times G, E) \to (M\times G, M)$ of action groupoids. Additionally, there is an equivalence of symmetric monoidal categories between the category of homogeneous vector bundles over the $G$-set $M$ and the category of linear representations of the groupoid $(M\times G, M)$\footnote{This perhaps suggests that certain spaces of motions could be better understood by appealing to the symmetric rigid monoidal category of linear representations of finite type over adequate groupoids.}. Furthermore, the functor of global sections can be identified with the induction functor attached to the canonical morphism of groupoids
$(M\times G, M) \to (G,\{*\})$ (here the group $G$ is considered as a groupoid with only one object~$\{*\}$, see Example~\ref{exam:GM} below). As we will see, in the groupoid context, the induction functor is related to the restriction functor via the (right) Frobenius reciprocity formula.

Frobenius reciprocity formula appears in the framework of finite groups under different forms (see for instance \cite[equation~(3.4), p.~109]{Sternberg:1994} or \cite[equation~(3.7), p.~111]{Sternberg:1994} and, e.g., \cite[Proposition~2.3.9]{Kowalski:2014}\footnote{We refer to \cite[Section~3.6, p.~128]{Sternberg:1994} for an application of this formula to Raman spectrum in quantum mechanics.}) and it has been extended to other classes as well, like locally compact groups \cite{Mackey:1953, MooreC:1962} or certain algebraic groups \cite{Howe:1971}. In the finite case, this formula compares the vector space dimensions of homomorphism spaces of linear representations over two different groups connected by a morphism of groups. In more conceptual terms, this amounts to say that for a given morphism of groups (not necessarily finite), the restriction
functor has the induction functor as right adjoint and the co-induction functor as a left adjoint. From a categorical point of view, these functors are well known constructions due to Kan and termed right and left Kan extensions, respectively~\cite{MacLane}. In the same direction, if both groups are finite and the connecting morphism is injective, then the induction and co-induction functors are naturally isomorphic and the resulting morphism between the group algebras produces a Frobenius extension of unital algebras \cite{Kadison:1999} (this result becomes in fact a direct consequence of our main theorem, see the forthcoming subsection).

Apart from the interest they generate in algebra, geometry and topology, Frobenius unital algebras are objects that deserve to be studied on their own. For instance, commutative Frobenius algebras over fields, like group algebras of finite abelian groups, play a prominent role in 2-dimensional topological quantum field theory, as it was corroborated in \cite{Kock:2003}.

So far, we have been dealing with situations where only finitely many objects were available. In other words, Frobenius unital algebras and groups (or finite bundles of these ones) are objects mainly built from categories with finitely many objects. Up to our knowledge, the general case of infinitely many objects is still unexplored in the literature. As an illustration, the Frobenius formulae for locally compact topological groupoids are far from being understood, since these formulae are not even explicitly computed for the case of abstract groupoids.

\looseness=-1 Our motivation is to introduce the main ideas that underpin techniques from the theory of linear representation of groupoids in relation with their non-unital algebras, which however admits enough orthogonal idempotents (see the subsequent subsection for the definition). Thus, this paper intends to set up, in a very elementary way, the basic tools to establish Frobenius formulae in the context of abstract groupoids and to employ these formulae to characterize Frobenius extensions of groupoids on the one hand and Frobenius extensions of their associated path algebras on the other, hoping in this way to fill in the lack that is present in the literature about this subject.

\vspace{-1.2mm}

\subsection{General notions and notations}\label{ssec:notation}

We fix some conventions that will be held all along this paper. If $\mathcal{C}$ is a small category (the class of objects is actually a set) and $\cD$ is any other category, then the symbol $[\mathcal{C}, \cD]$ stands for the category whose objects are functors and whose morphisms are natural transformations between these. Since $\mathcal{C}$ is a small category, the resulting category is in fact a Hom-set category, which means that the class of morphisms (or arrows) between any pair of objects forms a set and this set will be denoted by ${\rm Nat}(F,G)$ for any pair of functors~$F$,~$G$. We shall represent a functor $F \colon \mathcal{C} \to \cD$ between small Hom-set categories as a pair of maps $F=(\Fa,\Fo)$, where $\Fa\colon \Ca \to \Da$ and $\Fo\colon \Co \to \Do$ are the associated maps on the sets of arrows and objects, respectively. Given two objects $d,d' \in \cD$, we denote as usual by $\cD(d,d')$ the set of all arrows from $d$ to $d'$. Assume now that we have a functor $F\colon \mathcal{C} \to \cD$. Then by $\cD(d,\Fa(f))$ we denote the map which sends any arrow $p \in \cD(d, F_{\Sscript{0}}(s(f)))$ to the composition $\Fa(f)p \in \cD(d,F_{\Sscript{0}}(t(f)))$ (here $s(h)$ and $t(h)$ stand for the source and the target of a given arrow $h$). In this way, for each object $d \in \cD$ we have a functor $\cD(d,F(-))$ from $\mathcal{C}$ to the category of sets. Similarly, for each object $d' \in \cD$, we have the functor $\cD(F(-), d')$, as well as the functor $\cD(F(-), F(-))$ from the category $\mathcal{C}^{\Sscript{\rm op}} \times \mathcal{C}$ to the category of sets ($\mathcal{C}^{\Sscript{\rm op}}$ is the opposite category of $\mathcal{C}$ obtained by reversing the arrows of~$\mathcal{C}$).

Let $\Bbbk$ be a fixed base field and $1_{\Sscript{\Bbbk}}$ its identity element. Vector spaces over $\Bbbk$ and their morphisms (i.e.,~$\Bbbk$-linear maps) form a category, denoted by $\Vect$. Finite dimensional ones form a full subcategory of this, denoted by $\vect$. The symbol $\tensor{\Bbbk}$ denotes the tensor product between $\Bbbk$-vector spaces and their $\Bbbk$-linear maps. For any set $S$, we denote by $\Bbbk S:={\rm Span}_{\Sscript{\Bbbk}}\big\{x \,|\, x \in S \big\}$ the $\Bbbk$-vector space whose basis is the set $S$. Any element $x \in S$ is identified with its image $1_{\Sscript{\Bbbk}} x \in \Bbbk S$. By convention $\Bbbk S$ is the zero vector space whenever~$S$ is an empty set. When it is needed, we will also consider~$\Bbbk S$ as a set.

\looseness=-1 In this paper we shall consider rings without identity element (i.e., unity). Nevertheless, we will consider a class of rings (or $\Bbbk$-algebras) which have enough orthogonal idempotents in the sense of \cite{Fuller:1976, Gabriel:1962}, and that are mainly constructed from small categories. Specifically, given any small Hom-set category $\cD$, we can consider the \emph{path algebra} or \emph{Gabriel's ring of $\cD$}: Its underlying $\Bbbk$-vector space is the direct sum $R=\bigoplus_{\Sscript{x, x' \in \Do}} \Bbbk \cD(x,x')$ of $\Bbbk$-vector spaces. The multiplication of this ring is given by
the composition of $\cD$. Thus, for any two homogeneous generic elements $r, r' \in \Da$, the multiplication $(1_{\Sscript{\Bbbk}}r) . (1_{\Sscript{\Bbbk}}r')$ is defined by the rule: $ (1_{\Sscript{\Bbbk}}r) . (1_{\Sscript{\Bbbk}}r') =1_{\Sscript{\Bbbk}} (r r') $, the image of the composition of $r$ and $r'$ when $s(r)=t(r')$, otherwise we set $ (1_{\Sscript{\Bbbk}}r) . (1_{\Sscript{\Bbbk}}r') =0$ (see \cite[p.~346]{Gabriel:1962}). For any $x \in \Do$ we denote by $1_{\Sscript{x}}$ the image of the identity arrow of $x$ in the $\Bbbk$-vector space $R$.

In general, the ring $R$ has no unit, unless the set of objects $\Do$ is finite. Instead of that, it has a set of local units\footnote{Recall that a $\Bbbk$-vector space $R$ endowed with an associative $\Bbbk$-bilinear multiplication is said to be a \emph{ring with local units} over $\Bbbk$ if it has a set of idempotent elements, say $E \subset R$, such that for any finite subset of elements $\{r_1, \dots,r_n\} \subset R$, there is an element $e \in E$ such that $r_ie=er_i=r_i$, for any $i=1,\dots,n$. This means that any two elements $r, r' \in R$ are contained in an unital subring of the form $R_e=eRe$, for some $e \in E$, and $R$ is a~directed limit of the $R_e$'s, see \cite{Abrams:1983, Anh/Marki:1983, Anh/Marki:1987} and~\cite{ElKaoutit:2009}.}. Namely, the local units are given by the set of idempotent elements:
\begin{gather*}
\big\{1_{\Sscript{x_1}} \dotplus \cdots \dotplus 1_{\Sscript{x_n}}\in R \,|\, x_i \in\Do,\, i=1,\dots, n, \text{ and } n \in \mathbb{N}\setminus\{0\}\big\}.
\end{gather*}
For example, if we assume that $\cD$ is a discrete category, that is, the only arrows are the identities (or equivalently $\Do=\Da=X$) then $R=\Bbbk^{(X)}$ is the ring defined as the direct sum of $X$-copies of the base field.

A \emph{unital right $R$-module} is a right $R$-module $M$ such that $MR=M$, \emph{left unital modules} are similarly defined (see \cite[p.~347]{Gabriel:1962}). For instance, the previous ring $R$ attached to the category~$\cD$ decomposes as a direct sum of left and also of right unital $R$-modules:
\begin{gather*}
R = \bigoplus_{x \in \Do} R 1_{\Sscript{x}} = \bigoplus_{x \in \Do} 1_{\Sscript{x}} R.
\end{gather*}
Following \cite{Fuller:1976}, a ring which satisfies these two equalities is referred to as a \emph{ring with enough orthogonal idempotents}, whose complete set of idempotents is the set $\{1_{\Sscript{x}}\}_{x \in \Do}$. A morphism in this category of rings is obviously defined.

The $\Bbbk$-vector space of all homomorphisms between two left $R$-modules $M$ and $M'$ will be denoted by $\hom{R\text{-}}{M}{M'}$. If $T$ is another ring with enough orthogonal idempotents and if $M$ is an $(R, T)$-bimodule ($R$ acts on the left and $T$ on the right), then $\hom{R\text{-}}{M}{M'}$ is considered as left $T$-module by using the standard action $(a .f) \colon M \to M'$ sending $m$ to $f(m a)$ for every $a \in T$ and $f \in \hom{R\text{-}}{M}{M'}$.

\subsection{Description of the main results}\label{ssec:raroes}
A \emph{groupoid} is a small Hom-set category where each morphism is an isomorphism. More precisely, this is a pair of sets
$\cG:=(\cG_{\Sscript{1}}, \cG_{\Sscript{0}})$ with diagram of sets
\begin{gather*}
\xymatrix@C=35pt{
\cG_{\Sscript{1}}\ar@<0.70ex>@{->}|-{\scriptscriptstyle{{s}}}[r]
\ar@<-0.70ex>@{->}|-{\scriptscriptstyle{{t}}}[r] & \ar@{->}|-{
\scriptscriptstyle{\iota}}[l] \cG_{\Sscript{0}}},
\end{gather*}
where as above $s$ and $t$ are the source and the target of a given arrow respectively, and $\iota$ assigns to each object its identity arrow. In addition, there is an associative and unital multiplication $\cG_{\Sscript{2}}:= \cG_{\Sscript{1}} \due \times {\Sscript{{s}}} { \Sscript{{t}}} \cG_{\Sscript{1}} \to \cG_{\Sscript{1}}$ acting by $(f,g) \mapsto fg$, as well as a map $\cG_{\Sscript{1}} \to \cG_{\Sscript{1}}$ which associates to each arrow its inverse. Notice that $\iota$ is an injective map, and so~$\Go$ is identified with a subset of~$\Ga$. Then a groupoid is a category with additional structure, namely the map which sends any arrow to its inverse. We implicitly identify a groupoid with its underlying category. A morphism between two groupoids is just a functor between the underlying categories.

Given a groupoid $\cG$, we denote by $\Rep{\cG}$ its category of $\Bbbk$-linear representations, that is, $\Rep{\cG}=[\cG, \Vect]$, the category of functors from $\cG$ to $\Vect$. Let $\upphi \colon \cH \to \cG$ be a morphism of groupoids and denote by $\upphi_{\Sscript{*}}\colon \Rep{\cG} \to \Rep{\cH}$ the associated restriction functor. The induction and the co-induction functors are denoted by~$\upphi^{*}$ and~${}^{*}\upphi$, respectively (see Lemmas~\ref{lema:Induction} and~\ref{lema:Co-Induction} for the precise definitions of these functors). These are the right and the left adjoint functor of $\upphi_{\Sscript{*}}$. We say that $\upphi$ is a \emph{Frobenius morphism} (see Definition~\ref{def:Frobenius} below) provided that~$\upphi^{*}$ and~${}^{*}\upphi$ are naturally isomorphic.

\looseness=-1 Let $\upphi\colon \cH \to \cG$ be a morphism as above and denote by $A$ and $B$ the rings with enough orthogonal idempotents attached to $\cH$ and $\cG$, respectively. Then there is a $\Bbbk$-linear map given by
\begin{gather}\label{Eq:phi}
\phi\colon \ A \longrightarrow B, \qquad \bigg( \sum_i \lambda_i h_i \longmapsto \sum_i \lambda_i \phia(h_i)\bigg).
\end{gather}

As it is shown in Example~\ref{exam:Phi} below, this map is not in general multiplicative. However, under the assumption that $\phio\colon \Ho \to \Go$ is an injective map, $\phi\colon A \to B$ becomes a homomorphism (or extension) of rings with enough orthogonal idempotents and the complete set of idempotents $\{1_{\Sscript{\phio(u)}}\}_{u \in \Ho}$ is injected into the set $\{1_{\Sscript{x}}\}_{x \in \Go}$. In this way, $B$ becomes an $A$-bimodule via the restriction of scalars, although not necessarily a unital one.

The notion of right (left) groupoid-set and that of groupoid-biset are explicitly recalled in Definitions~\ref{def:Gset} and \ref{def:biset}, respectively. For any morphism $\upphi$ of groupoids as above, we denote $\upg:=\Ga \due \times {\Sscript{s}}{ \Sscript{\upphi_0} } \Ho$ and similarly ${ }^{\Sscript{\upphi}}\uU(\cG):=\Ho \due \times{\Sscript{\upphi_0}}{ \Sscript{t}} \Ga$. These are the right and the left pull-back bisets associated to $\upphi$. More precisely, $\upg$ is the right pull-back $(\cG,\cH)$-biset with structure maps $\varsigma\colon \upg \to \Go$, $(a,u) \mapsto t(a)$ and ${\rm pr}_{\Sscript{2}}\colon \upg \to \Ho$, $(a,u) \mapsto u$. Its right $\cH$-action is given by $(a,u) h= (a\phia(h), s(h))$, whenever $u=t(h)$, and its left $\cG$-action is given by $g(a,u)=(ga,u)$, whenever $s(g)=t(a)$. Similarly, we find that the left pull-back ${ }^{\Sscript{\upphi}}\uU(\cG)$ is an $(\cH,\cG)$-biset with structure maps $\vartheta\colon {}^{\Sscript{\upphi}}\uU(\cG) \to \Go$, $(u,a) \mapsto s(a)$, ${\rm pr}_{\Sscript{1}}\colon \pug \to \Ho, (u,a) \mapsto u$. The left $\cH$-action is given by $h (u,a)=(t(h), \phia(h)a)$, where $s(h)=u$, while the right $\cG$-action is given by $(u,a) g=(u,ag)$, where $s(a)=t(g)$ (see Examples~\ref{exam: HG} and~\ref{exam:bisets} below).

Now that we have all the ingredients at our disposal, we proceed to articulate our main result in the subsequent theorem, which is stated below as Theorem~\ref{thm:FM}.

\begin{Theorem}\label{thmA:FM}Let $\upphi\colon \cH \to \cG$ be a morphism of groupoids and consider as above the associated rings $A$ and $B$, respectively. Assume that $\upphi_{\Sscript{0}}$ is an injective map. Then the following are equivalent.
\begin{enumerate}[$(i)$]\itemsep=0pt
\item $\upphi$ is a Frobenius morphism;
\item There exists a natural transformation $\Sf{E}_{\Sscript{(u, v)}}\colon \cG(\phio(u),\phio(v)) \longrightarrow \Bbbk \cH(u,v)$ in $\cH^{\Sscript{\rm op}} \times \cH$, and for every $x \in \Go$, there exists a finite set $\{\big((u_{i},b_{i}), c_{i}\big)\}_{i=1, \dots, N} \in \varsigma^{-1}\big( \{x\}\big)\times \Bbbk \cG(x, \phio(u_{i}))$ such that, for every pair of elements $ (b, b') \in \cG(x, \phio(u)) \times \cG(\phio(u), x)$, we have
\begin{gather*}
\sum_{i}\Sf{E}(bb_{i})c_{i} = b \in \Bbbk \cG(x, \phio(u)) \qquad \text{and} \qquad b' = \sum_{i}b_{i}\Sf{E}(c_{i}b') \in \Bbbk \cG(\phio(u), x).
\end{gather*}
\item For every $x\in \Go$, the left unital $A$-module $AB1_{x}$ is finitely generated and projective and there is a natural isomorphism $B1_{u} \cong B\hom{A-}{AB}{A1_{u}}$, of left
unital $B$-modules, for every $u\in \Ho$.
\end{enumerate}
\end{Theorem}

In what follows a groupoid is said to be \emph{finite} if it has finitely many connected components and each of its isotropy group is finite. On the other hand, it is noteworthy to mention that if the arrow map $\phia\colon \Ha \to \Ga$ of a given morphism of groupoids $\upphi\colon \cH \to \cG$ is injective, then $\upphi$ is obviously a faithful functor and the object map $\phio\colon \Ho \to \Go$ is also injective. The following result, which characterizes the case of an extension by subgroupoids, is a corollary of Theorem~\ref{thmA:FM} and stated below as Corollary~\ref{coro:Finite}.

\begin{Corollary}\label{coroB:Finite}Let $\upphi\colon \cH \to \cG$ be a morphism of groupoids such that $\upphi$ is a faithful functor and $\upphi_{\Sscript{0}}\colon \Ho \to \Go$ is an injective map. Then the following are equivalent:
\begin{enumerate}[$(a)$]\itemsep=0pt
\item $\upphi$ is a Frobenius extension;
\item For any $x \in \Go$, the left $\cH$-set $\vartheta^{-1}\big(\{x\} \big)$ has finitely many orbits;
\item For any $x \in \Go$, the right $\cH$-set $\varsigma^{-1}\big(\{x\} \big)$ has finitely many orbits.
\end{enumerate}
In particular, any inclusion of finite groupoids is a Frobenius extension.
\end{Corollary}

If we consider, in both Theorem~\ref{thmA:FM} and Corollary~\ref{coroB:Finite}, groupoids with only one object, then we recover the classical result on Frobenius extensions of group algebras, see Corollary \ref{cor:FiniteI} below for more details.

\section[Abstract groupoids: General definition, basic properties and examples]{Abstract groupoids: General definition, basic properties\\ and examples}\label{sec:Grpd}
This section contains all the material: definitions, properties and examples of abstract groupoids that will be used in the course of the following sections. This material was recollected form \cite{ElKaoutit:2015, Kaoutit/Kowalzig:14, Kaoutit/Spinosa:2018} and from the references quoted therein. All groupoids discussed below are abstract and small ones, in the sense that the class of arrows is actually a set, and they do not admit any topological or combinatorial structures.

\subsection{Notations, basic notions and examples}\label{ssec:basic}Let $\cG$ be a groupoid and consider an object $x \in \Go$. \emph{The isotropy group of $\cG$ at $x$}, is the group:
\begin{gather}\label{Eq:isotropy}
\cG^{\Sscript{x}}:=\cG(x,x) = \big\{ g \in \cG_{\Sscript{1}}\,|\, s(g)=t(g)=x \big\}.
\end{gather}

Clearly, for any two objects $x, y \in G_{\Sscript{0}}$, we have that each of the sets $\cG(x,y)$ is, by the groupoid multiplication, a right $\cG^{\Sscript{x}}$-set and left $\cG^{\Sscript{y}}$-set. Notice here that the multiplication of the groupoid is by convention defined as the map $\cG_{\Sscript{1}} \due \times {\Sscript{{s}}} { \Sscript{{t}}} \cG_{\Sscript{1}} \to \cG_{\Sscript{1}}$ sending $(f,g)$ to $f \circ g:=fg$, so that $\mathcal{G}^{y} \times \mathcal{G}(x,y)$ and $\mathcal{G}(x,y) \times \mathcal{G}^{x} $ are both subsets of $\cG_{\Sscript{1}} \due \times {\Sscript{{s}}} { \Sscript{{t}}} \cG_{\Sscript{1}} $. In fact, each of the $\cG(x,y)$'s is a $(\cG^{\Sscript{y}}, \cG^{\Sscript{x}})$-biset, see~\cite{Bouc:2010} for pertinent definitions.

The \emph{$($left$)$ star} of an object $x \in \Go$ is defined by $\Starl{x}:=t^{-1}\big( \{ x\}\big)=\{ g \in \Ga \,|\, t(g)=x \}$. The right star is defined using the source map, and both left and right stars are in bijection. Now, given an arrow $g \in \Ga$, we define the \emph{conjugation operation} (or the \emph{adjoint operator}) as the morphism of groups:
\begin{gather}\label{Eq:Ad}
\Ad{g}\colon \ \cG^{\Sscript{s(g)}} \longrightarrow \cG^{\Sscript{t(g)}}, \qquad \big( f \longmapsto gfg^{-1} \big).
\end{gather}

Let $\upphi\colon \cH \to \cG$ be a morphism of groupoids. Obviously $\upphi$ induces homomorphisms of groups between the isotropy groups: $\upphi^{\Sscript{y}}\colon \cH^{\Sscript{y}} \to \cG^{\Sscript{\upphi_0(y)}}$, for every $y \in H_{\Sscript{0}}$. The family of homomorphisms $\{\upphi^{\Sscript{y}}\}_{ y \in \Ho }$ is referred to as \emph{the isotropy maps of $\upphi$}. For a fixed object $x \in \Go$, its fibre $\phio^{-1}(\{x\})$, if not empty, leads to the following ``star'' of homomorphisms of groups:
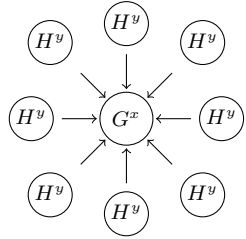
\begin{scriptsize}
\begin{center}
\tikzstyle{every pin edge}=[<-,shorten <=1pt]
\tikz[pin distance=5mm]
\node [circle, draw, pin=right: {\mycircled{$H^{\Sscript{y}}$}},
pin=above right:{\mycircled{$H^{\Sscript{y}}$}},
pin=above:{\mycircled{$H^{\Sscript{y}}$}}, pin=above
left:{\mycircled{$H^{\Sscript{y}}$}},
pin=left:{\mycircled{$H^{\Sscript{y}}$}},
pin=below:{\mycircled{$H^{\Sscript{y}}$}}, pin=below
left:{\mycircled{$H^{\Sscript{y}}$}},pin=below
right:{\mycircled{$H^{\Sscript{y}}$}} ]
{$G^{\Sscript{x}}$};
\end{center}
\end{scriptsize}
where $y$ runs in the fibre $\phio^{-1}(\{x\})$.

\begin{Example}[trivial groupoid]\label{exam:trivial}Let $X$ be any set. Then the pair $(X, X)$ can be considered as a groupoid with a trivial structure. Thus, the only arrows are the identities. This groupoid is known as \emph{a trivial groupoid}.
\end{Example}

\begin{Example}[action groupoid]\label{exam:action} \looseness=-1 Any group $G$ can be considered as a~groupoid by taking $G_{\Sscript{1}}=G$ and $G_{\Sscript{0}}=\{*\}$ (a set with one element). Now, if $X$ is a right $G$-set with action $\rho\colon X\times G \to X$, then one can define the so called \emph{the action groupoid}: $G_{\Sscript{1}}=X \times G$ and $G_{\Sscript{0}}=X$, the source and the target are $s=\rho$ and $t={\rm pr}_{\Sscript{1}}$, the identity map sends $x \mapsto (x, e)=\iota_{\Sscript{x}}$, where $e$ is the identity element of $G$. The multiplication is given by $(x,g) (x',g')=(x,gg')$, whenever $xg=x'$, and the inverse is defined by $(x,g)^{-1}=\big(xg,g^{-1}\big)$. Clearly the pair of maps $({\rm pr}_{\Sscript{2}}, *)\colon \cG:= (G_{\Sscript{1}}, G_{\Sscript{0}}) \to (G,\{*\})$ defines a morphism of groupoids. For a given $ x \in X$, the isotropy group $\cG^{\Sscript{x}}$ is clearly identified with the stabilizer $\emph{Stab}_{\Sscript{G}}(x)=\{g \in G \,|\, gx=x\}$ subgroup of~$G$.

Given $\sigma\colon X \to Y$ a morphism of right $G$-sets, then the pair of maps $(X\times G, X) \to (Y\times G, Y)$ sending $\big( (x,g), x' \big) \mapsto \big( (\sigma(x), g),\sigma(x') \big)$ defines a morphism of action groupoids.
\end{Example}

\begin{Example}[equivalence relation groupoid]\label{exam:X} We expound here several examples which, in fact, belong to the same class, that of equivalence relation groupoids.
\begin{enumerate}[(1)]\itemsep=0pt
\item One can associate to a given set $X$ the so called \emph{the groupoid of pairs} (called \emph{fine groupoid} in~\cite{Brown:1987} and \emph{simplicial groupoid} in~\cite{Higgins:1971}), its set of arrows is defined by $G_{\Sscript{1}}=X \times X$ and the set of objects by $G_{\Sscript{0}}=X$; the source and the target are $s={\rm pr}_{\Sscript{2}}$ and $t={\rm pr}_{\Sscript{1}}$, the second and the first projections, and the map of identity arrows $\iota$ is the diagonal map. The multiplication and the inverse maps are given by
\begin{gather*}
(x,x') (x',x'') = (x,x''),\qquad \text{and} \qquad (x,x')^{-1} = (x',x).
\end{gather*}
Let $f\colon X\to Y$ be any map and consider the trivial groupoid $(X,X)$ as in Example~\ref{exam:trivial} together with the groupoid of pairs $(Y \times Y, Y)$. Then, the pair of maps
$(F_{\Sscript{1}}, F_{\Sscript{0}})\colon (X,X) \to (Y\times Y, Y)$, where $F_{\Sscript{1}}\colon X \to Y\times Y$, $x \mapsto (f(x), f(x))$ and $F_{\Sscript{0}}=f$, establishes a morphism of groupoids.

\item Let $\upnu\colon X \to Y$ be a map. Consider the fibre product $X \due \times {\Sscript{\upnu}} {\; \Sscript{\upnu}} X$ as a set of arrows of the groupoid $\xymatrix@C=35pt{ X \due \times {\Sscript{\upnu}} {\; \Sscript{\upnu}} X \ar@<0.8ex>@{->}|-{\scriptscriptstyle{{\rm pr}_2}}[r] \ar@<-0.8ex>@{->}|-{\scriptscriptstyle{{\rm pr}_1}}[r] & \ar@{->}|-{ \scriptscriptstyle{\iota}}[l] X, }$ where as before $s={\rm pr}_{\Sscript{2}}$ and $t={\rm pr}_{\Sscript{1}}$, and the map of identities arrows is $\iota$ the diagonal map. The multiplication and the inverse are the obvious ones.

\item Assume that $\cR \subseteq X \times X$ is an equivalence relation on the set $X$. One can construct a groupoid
$\xymatrix@C=35pt{\cR \ar@<0.8ex>@{->}|-{\scriptscriptstyle{{\rm pr}_2}}[r] \ar@<-0.8ex>@{->}|-{\scriptscriptstyle{{\rm pr}_1}}[r] & \ar@{->}|-{\scriptscriptstyle{\iota}}[l] X, }$ with structure maps as before. This is an important class of groupoids known as \emph{the groupoid of equivalence relation} (or \emph{equivalence relation groupoid}). Obviously $(\cR, X) \hookrightarrow (X\times X, X)$ is a morphism of groupoid, see for instance \cite[Example~1.4, p.~301]{DemGab:GATIGAGGC}.
\end{enumerate}
Notice that in all these examples each of the isotropy groups is the trivial group.
\end{Example}

\begin{Example}[induced groupoid]\label{exam:induced}Let $\cG=(G_{\Sscript{1}}, G_{\Sscript{0}})$ be a groupoid and $\varsigma\colon X \to G_{\Sscript{0}}$ a map. Consider the following pair of sets:
\begin{gather*}
G^{\Sscript{\varsigma}}{}_{\Sscript{1}}:= X \due \times {\Sscript{\varsigma}} { \Sscript{t}} G_{\Sscript{1}} \; \due \times {\Sscript{s}} { \Sscript{\varsigma}} X= \big\{ (x,g,x')
\in X\times G_{\Sscript{1}}\times X\,| \, \varsigma(x)=t(g), \varsigma(x')=s(g) \big\}, \quad G^{\Sscript{\varsigma}}{}_{\Sscript{0}}:=X.
\end{gather*}
Then $\cG^{\Sscript{\varsigma}}{}=(G^{\Sscript{\varsigma}}{}_{\Sscript{1}}, G^{\Sscript{\varsigma}}{}_{\Sscript{0}})$ is a groupoid, with structure maps: $s= {\rm pr}_{\Sscript{3}}$, $t= {\rm pr}_{\Sscript{1}}$, $\iota_{\Sscript{x}}=(\varsigma(x), \iota_{\Sscript{\varsigma(x)}}, \varsigma(x))$, $x \in X$. The multiplication is defined by $(x,g,y) (x',g',y')= ( x,gg',y')$, whenever $y=x'$, and the inverse is given by $(x,g,y)^{-1}=\big(y,g^{-1},x\big)$. The groupoid $\cG^{\Sscript{\varsigma}}$ is known as \emph{the induced groupoid of $\cG$ by the map~$\varsigma$}, (or \emph{the
pull-back groupoid of $\cG$ along $\varsigma$}, see~\cite{Higgins:1971} for dual notion). Clearly, there is a canonical morphism $\phi^{\Sscript{\varsigma}}:=({\rm pr}_{\Sscript{2}},
\varsigma)\colon \cG^{\Sscript{\varsigma}} \to \cG$ of groupoids. A particular instance of an induced groupoid, is the one when $\cG=G$ is a groupoid with one object. Thus, for any group~$G$, one can consider the Cartesian product $X \times G \times X$ as a set of arrows of a groupoid with set of objects~$X$.
\end{Example}

\begin{Example}[frame groupoid]\label{exam:frame}Let $\pi\colon \cY \to X$ be a surjective map, and write $\cY =\biguplus_{\Sscript{ x \in X }} \cY_{\Sscript{x}}$, where $\cY_{\Sscript{x}}:=\pi^{-1}(\{x\})$ (the fibres of $\pi$ at $x$). For any pair of elements $x, x' \in X$, we set
\begin{gather*}
\cG(x,x'):=\big\{ f \colon \cY_{\Sscript{x}} \to \cY_{\Sscript{x'}}\,|\, f \text{ is a bijective map} \big\},
\end{gather*}
then the pair $(\Ga, \Go):=\big(\biguplus_{\Sscript{x, x' \in X}} \cG(x,x'), X\big)$ admits a structure of groupoid (possibly a trivial one), referred to as \emph{the frame groupoid of $(\cY, \pi)$} and denoted by ${\rm Iso}(\cY, \pi)$, see also~\cite{Renault:1980}.

In a more general setting, one can similarly define the frame groupoid of a given family $\{\cY_{\Sscript{x}}\}_{x \in X}$ of objects in a certain category, indexed by a set~$X$. For instance, we could take each of the $\cY_{\Sscript{x}}$'s as an abelian group (resp.~$\Bbbk$-vector space), in this case, the set of arrows $\cG(x,x')$ should be the set of all abelian group isomorphisms (resp.~$\Bbbk$-linear isomorphisms).
\end{Example}

\begin{Example}[isotropy groupoid]\label{exam:IsotropyGrpd}Let $\cG$ be a groupoid, then the disjoint union ${\biguplus}_{\Sscript{x \in G_0}} \cG^{\Sscript{x}}$ of all its isotropy groups form the set of arrows of a subgroupoid of $\cG$ whose source is equal to its target, namely the projection $\varsigma\colon {\biguplus}_{\Sscript{x \in G_0}} \cG^{\Sscript{x}} \to G_{\Sscript{0}}$. We denote this groupoid by $\cG^{\Sscript{(i)}}$ and refer to it as \emph{the isotropy groupoid of~$\cG$}. For instance, the isotropy groupoid of any equivalence relation groupoid is a trivial one as in Example~\ref{exam:trivial}.
\end{Example}

\subsection{Groupoids actions and equivariant maps}\label{ssec:Grpd1}
The subsequent definition is, in fact, an abstract formulation of that given in \cite[Definition~1.6.1]{Mackenzie:2005} for Lie groupoids, and essentially the same definition based on the Sets-bundles notion given in \cite[Definition~1.11]{Renault:1980}.

\begin{Definition}\label{def:Gset}Given a groupoid $\mathcal{G}$ and a map $\varsigma\colon X \to \Go$. We say that $(X,\varsigma)$ is a \emph{right} $\cG$-\emph{set} (with a \emph{structure map} $\varsigma$), if there is a map (\emph{the action}) $\rho\colon X \due \times {\Sscript{\varsigma}} { \Sscript{{t}}} \Ga \to X$ sending $(x,g) \mapsto xg$, satisfying the following conditions:
\begin{enumerate}[1)]\itemsep=0pt
\item $s(g)=\varsigma(xg)$, for any $x \in X$ and $g \in \Ga$ with $\varsigma(x)=t(g)$,
\item $x \iota_{\varsigma(x)}= x$, for every $x \in X$,
\item $ (xg)h= x(gh)$, for every $x \in X$, $g,h \in \Ga$ with $\varsigma(x)={t}(g)$ and $t(h)=s(g)$.
\end{enumerate}
\end{Definition}

A \emph{left action} is analogously defined by interchanging the source with the target. In general, a set with a (right or left) groupoid action is called \emph{a groupoid-set}.

\begin{Remark}\label{rem:GrpdSets}If we think of group as a groupoid with a single object, then Definition~\ref{def:Gset} leads to the definition of the usual action of a group on a set (see~\cite{Bouc:2010}). From a categorical point of view, this action is nothing but a functor from the underlying category of such a groupoid to the core category of sets\footnote{The core category of a given category is the subcategory whose arrows are all the isomorphisms.}. Writing down this formulation for groupoids with several objects, will leads to the Definition~\ref{def:Gset}. Specifically, following \cite[Remark~2.6]{Kaoutit/Spinosa:2018}, for any groupoid $\cG$, there is a (symmetric monoidal) equivalence between the category of right $\cG$-sets and the category of functors from $\cG^{\Sscript{\rm op}}$ to the core category of sets. An analogue equivalence of categories holds true for left $\cG$-sets. Following the same reasons that were explained in \cite[Remark 2.6, Section~5.3]{Kaoutit/Spinosa:2018}, in this paper we will work with Definition~\ref{def:Gset} instead of the aforementioned functorial approach.
\end{Remark}

Obviously, any groupoid $\cG$ acts over itself on both sides by using the regular action, i.e., the multiplication $\Ga \due \times {\Sscript{{s}}} { \Sscript{{t}}} \Ga \to \Ga$. That is, $(\Ga, {s})$ is a right $\cG$-set and $(\Ga, {t})$ is a left $\cG$-set with this action. On the other hand, the pair $(\Go, {\rm id})$ admits a structure of right $\cG$-set, as well as a structure of a left $\cG$-set. For instance, the right action is given by the map $\Go \due \times {\Sscript{{\rm id}}} { \Sscript{{t}}} \Ga \to \Go$ sending $(x,g) \mapsto x . g= s(g)$.

A \emph{morphism of right $\cG$-sets} (or \emph{$\cG$-equivariant map}) $F\colon (X,\varsigma) \to (X',\varsigma')$ is a map $F\colon X \to X'$ such that the diagrams
\begin{gather}
\begin{gathered}
\xymatrix@R=12pt{ & X \ar@{->}_-{\Sscript{\varsigma}}[ld]
\ar@{->}^-{F}[dd] & \\ G_{\Sscript{0}}& & \\ & X'
\ar@{->}^-{\Sscript{\varsigma'}}[lu] & } \qquad \qquad
\xymatrix@R=12pt{X \due \times {\Sscript{\varsigma}} {
\Sscript{t}} \Ga \ar@{->}^-{}[rr] \ar@{->}_-{\Sscript{F \times
 {\rm id}}}[dd] & & X \ar@{->}^-{\Sscript{F}}[dd] \\ & & \\ X' \due
\times {\Sscript{\varsigma'}} { \Sscript{t}} \Ga
\ar@{->}^-{}[rr] & & X' }
\end{gathered}
\end{gather}
commute. The category so is constructed is termed \emph{the category of right $\cG$-sets} and denoted by $\Gsets$. It is noteworthy to mention that this category admits a structure of symmetric monoidal category, which is isomorphic to the category of left $\cG$-sets. Indeed, to any right $\cG$-set $(X,\varsigma)$ one associated its \emph{opposite left $\cG$-set} $(X, \varsigma)^{\Sscript{o}}$ whose underlying set is~$X$ and structure maps is $\varsigma$, while the left action is given by $g x = x\big(g^{-1}\big)$, for every pair $(g, x) \in \Ga \due \times {\Sscript{s}} { \Sscript{\varsigma}} X$, see~\cite{Kaoutit/Kowalzig:14} for more properties of these categories.

Given two right $\cG$-sets $(X,\varsigma)$ and $(X',\varsigma')$, we denote by $\hom{\Gsets}{X}{X'}$ the set of all $\cG$-equivariant maps from $(X,\varsigma)$ to $(X',\varsigma')$. A subset $Y \subseteq X$ of a right $\cG$-set $(X,\varsigma)$, is said to be \emph{$\cG$-invariant} whenever the inclusion $Y \hookrightarrow X$ is a $\cG$-equivariant map. For instance, any left star $\Starl{x}$ of any object $x \in\Go$, is a $\cG$-invariant subset of the right $\cG$-set $(\Ga, s)$.

A trivial example of right groupoid-set is a right group-set. Specifically, if we consider a group as a groupoid with only one object, then its category of group-sets coincides with its category of groupoid-sets. The following example, which will be used in the sequel, describes non trivial examples of groupoids-sets.

\begin{Example}\label{exam: HG}Let $\upphi\colon \cH \!\to\! \cG$ be a morphism of groupoids. Consider the triple $(\Ho \due \times {\Sscript{\upphi_0}} { \Sscript{t}} \Ga, {\rm pr}_{\Sscript{1}}, \vartheta)$, where $\vartheta\colon \Ho \due \times {\Sscript{\upphi_0}} { \Sscript{t}} \Ga \to \Go$ sends $(u,a) \mapsto s(a)$, and ${\rm pr}_{\Sscript{1}}$ is the first projection. Then the following maps
\begin{gather*}
\xymatrix@R=0pt{ \big(\Ho \due \times {\Sscript{\upphi_0}} { \Sscript{t}} \Ga\big) \due \times {\Sscript{\vartheta}} { \Sscript{t}} \Ga \ar@{->}^-{}[r] & \Ho \due \times {\Sscript{\upphi_0}} { \Sscript{t}} \Ga, \\ \big((u,a),g\big)\ar@{|->}^-{}[r] & (u, ag), } \qquad \xymatrix@R=0pt{ \Ha \due \times {\Sscript{s}} { \Sscript{{\rm pr}_1}} \big(\Ho \due
\times {\Sscript{\upphi_0}} { \Sscript{t}} \Ga\big) \ar@{->}^-{}[r] & \Ho \due \times {\Sscript{\upphi_0}} { \Sscript{t}} \Ga, \\ \big(h, (u,a)\big) \ar@{|->}^-{}[r] & (t(h),\phia(h)a) }
\end{gather*}
define, respectively, a structure of right $\cG$-sets and that of left $\cH$-set. Analogously, the maps
\begin{gather*}
\xymatrix@R=0pt{ \big(\Ga \due \times {\Sscript{s}} { \Sscript{\upphi_0}} \Ho\big) \due \times {\Sscript{{\rm pr}_2}} { \Sscript{t}} \Ha \ar@{->}^-{}[r] & \Ga \due \times
{\Sscript{s}} { \Sscript{\upphi_0}} \Ho, \\ \big((a,u),h \big)\ar@{|->}^-{}[r] & (a\phia(h), s(h)), } \qquad
\xymatrix@R=0pt{ \Ga \due \times {\Sscript{s}} { \Sscript{\varsigma}} \big(\Ga \due \times {\Sscript{s}} { \Sscript{\upphi_0}} \Ha\big) \ar@{->}^-{}[r] & \Ga \due \times
{\Sscript{s}} { \Sscript{\upphi_0}} \Ho, \\ \big(g, (a,u)\big)\ar@{|->}^-{}[r] & (ga,u), }
\end{gather*}
where $\varsigma\colon \Ga \due \times {\Sscript{s}} { \Sscript{\upphi_0}} \Ho \to \Go$ sends $(a,u) \mapsto t(a)$, define, respectively, a right $\cH$-set and left $\cG$-set structures on $\Ga \due \times {\Sscript{s}} { \Sscript{\upphi_0}} \Ho$. This in particular can be applied to the morphisms described in Examples~\ref{exam:action} and~\ref{exam:X}(1). More precisely, keeping the notation of these two examples, then in the first one, we have that $X \due \times {\Sscript{\sigma}} { \Sscript{{\rm pr}_1}} (Y \times G) =\big\{ (x, (\sigma(x),g)) \,|\, x \in X, g \in G \big\} $ is a right groupoid-set with structure map $(x, (\sigma(x),g)) \mapsto \sigma(xg)$ and action $(x, (\sigma(x),g)) (y, g' )= (x, (\sigma(x),gg'))$, whenever $\sigma(xg)=y$. Moreover, this set is also a left groupoid-set with structure map $(x, (\sigma(x),g)) \mapsto x$ and action $(x',g') (x, (\sigma(x),g)) = (x', (\sigma(x'),g'g))$, whenever $\sigma(x'g')=\sigma(x)$. Concerning the second example, we have that the set $(Y\times Y) \due \times {\Sscript{{\rm pr}_1}} { \Sscript{f}} X =\big\{ (y,f(x)),x)| x \in X, y \in Y\big\}$ is a left groupoid-set with structure map sending $((y,f(x)),x) \mapsto y$ and action $(y',y) ((y, f(x)),x) = ((y',f(x)),x)$, while its right groupoid-set structure is the trivial one.
\end{Example}

\subsection{Translation groupoids and the orbits sets}\label{ssec:Grpd12}
\looseness=-1 Let $\cG$ be a groupoid and $(X,\varsigma)$ a right $\cG$-set. Consider the pair of sets $\big( X \due \times {\Sscript{\varsigma}} { \Sscript{{t}}} \Ga, X \big)$ as a~groupoid with structure maps ${s}=\rho$, $t={\rm pr}_{\Sscript{1}}$, $\iota_{\Sscript{x}}=(x,\iota_{\varsigma(x)})$. The multiplication and the inverse maps are defined by $(x,g)(x',g')=(x,gg')$ and $(x,g)^{-1}=\big(xg,g^{-1}\big)$. This groupoid is denoted by $X \rJoin \cG$ and it is known in the literature as the \emph{right translation groupoid of~$X$ by~$\cG$} (or \emph{semi-direct product groupoid}, see for instance \cite[p.~163]{Moedijk/Mrcun:2005} and \cite{Kaoutit/Kowalzig:14}). Furthermore, there is a canonical morphism of groupoids $\upsigma\colon X\rJoin \cG \to \cG$, given by the pair of maps $\upsigma=({\rm pr}_{\Sscript{2}}, \varsigma)$. Clearly any $\cG$-equivariant map $F\colon (X,\varsigma) \to (X',\varsigma')$, induces a morphism $\Sf{F}\colon X \rJoin \cG \to X' \rJoin \cG$ of groupoids, whose arrows map is given by $X \due \times {\Sscript{\varsigma}} { \Sscript{{t}}} \Ga \to X' \due \times {\Sscript{\varsigma'}} { \Sscript{{t}}} \Ga $, $(x, g) \mapsto (F(x), g)$, and its objects map is $F\colon X \to X'$.

\begin{Example}\label{exam:refereePacience}Any groupoid $\cG$ can be seen as a (right) translation groupoid of $\Go$ along $\cG$ itself. Thus, the right translation groupoid of the right $\cG$-set $(\Go, id)$ coincides (up to a canonical iso) with $\cG$ itself. Now, let $G$ be a group and $X$ a right $G$-set. Then the attached action groupoid, as described in Example~\ref{exam:action}, is precisely the right translation groupoid of~$X$ along $G$, where $G$ is considered as a groupoid with one object.
\end{Example}

Next we recall the notion of the orbit set attached to a right groupoid-set. This notion is a~generalization of the orbit set in the context of group-sets. Here we use the (right) translation groupoid to introduce this set.
First we recall the notion of the orbit set of a given groupoid. \emph{The orbit set of a groupoid} $\cG$ is the quotient set of $G_{\Sscript{0}}$ by the following equivalence relation: Two objects $x, x' \in G_{\Sscript{0}}$ are said to be equivalent if and only if there is an arrow connecting them, that is, there is $g \in\Ga$ such that $t(g)=x$ and $s(g)=x'$. Viewing $x, x' \in \Go$ as elements in the right (or left) $\cG$-set $(\Go, {\rm id})$, then this means that $x$ and $x'$ are equivalent if and only if, there exists $g \in \Ga$ such that $x . g=x'$. The quotient set of~$\Go$ by this equivalence relation, is nothing but the set of all connected components of $\cG$, which we denote by $\pi_{\Sscript{0}}(\cG):=\Go/\cG$.

Given a right $\cG$-set $(X,\varsigma)$, the \emph{orbit set} $X/\cG$ of $(X,\varsigma)$ is the orbit set of the (right) translation groupoid $X \rJoin \cG$, that is, $X/\cG=\pi_{\Sscript{0}}(X \rJoin \cG)$. If $\cG=(X\times G, X)$ is an action groupoid as in Example~\ref{exam:action}, then obviously the orbit set of this groupoid coincides with the classical set of orbits $X/G$, see also Example~\ref{exam:refereePacience}. Of course, the orbit set of an equivalence relation groupoid $(\cR, X)$, see Example~\ref{exam:X}, is precisely the quotient set $X/\cR$ modulo the equivalence relation $\cR$.

\begin{Remark}\label{rem:R} In this remark we exhibit the connection between a given groupoid and its attached equivalence relation groupoid. So, let $\cG$ be a groupoid and consider the pair of maps $\big( (t, s), {\rm id}\big)\colon (\Ga, \Go) \to (\Go \times \Go, \Go)$, where the first component sends $g \mapsto (t(g), s(g))$. The pair $\big( (t, s), {\rm id}\big)$ establishes a morphism of groupoids from $\cG$ to the groupoid of pairs $(\Go \times \Go, \Go)$. Now, denotes by $\cR$ the equivalence relation defined as above by the action of $\cG$ on $\Go$, that is, for a given pair of objects $ x, x' \in \Go$, we have that $x\sim_{\Sscript{\cR}} x'$, if and only if, there is an arrow $g \in \Ga$ such that $s(g)=x$ and $t(g)=x'$. In this way, we obtain another groupoid, namely, the equivalence relation groupoid $(\cR, \Go)$ as in Example~\ref{exam:X}(3).

These three groupoids are connected by the following commutative diagram of groupoids:
\begin{gather}\label{diag:not}
\begin{gathered}
\xymatrix@R=12pt{ & & \Ga \ar@{->}_-{\Sscript{(t,s)}}[lld] \ar@{->>}^{}[dd] |!{[dll];[drr]} \hole
\ar@/^1pc/@<0.80ex>@{->}|-{\scriptscriptstyle{{s}}}[rrd]
\ar@/^1pc/@<-0.80ex>@{->}|-{\scriptscriptstyle{{t}}}[rrd] & & \\
\Go \times \Go \ar@<0.80ex>@{->}|-{\scriptscriptstyle{{s}}}[rrrr]
\ar@<-0.80ex>@{->}|-{\scriptscriptstyle{{t}}}[rrrr] & & & &
\ar@{->}|-{\scriptscriptstyle{{\iota}}}[llll] \Go.
\ar@/^1pc/@{->}|-{\scriptscriptstyle{{\iota}}}[dll]
\ar@/_1pc/@{->}|-{\scriptscriptstyle{{\iota}}}[ull] \\ & & \cR
\ar@/_1pc/@<0.80ex>@{->}|-{\scriptscriptstyle{{s}}}[rru]
\ar@/_1pc/@<-0.80ex>@{->}|-{\scriptscriptstyle{{t}}}[rru]
\ar@{_{(}->}[llu] & & }
\end{gathered}
\end{gather}
More precisely, we already know from Example \ref{exam:X}(3) that there is a morphism of groupoids $(\cR, \Go) \to (\Go \times \Go, \Go)$, and since the image of $(t,s)$ lands in $\cR$, we obtain the vertical morphism of groupoids, whose arrow map is by definition surjective. If in diagram \eqref{diag:not} the lower left hand map is an identity, i.e., if $\cR=\Go \times \Go$, then $\cG$ posses only one connected component. Thus~$\pi_{\Sscript{0}}(\cG)$ is a set with one element, and this happens if and only if $\cG$ is a \emph{transitive groupoid}.

Summing up, the vertical map in diagram~\eqref{diag:not} is injective, if and only if, $\cG \cong (\cR, \Go)$ an isomorphism of groupoids, if and only if, $\cG$ has no parallel arrows, that is, none of the forms

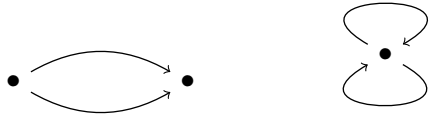
\begin{center}
\begin{tikzpicture}[x=11pt,y=11pt,thick]\pgfsetlinewidth{0.5pt}
\node(1) at (-6,0){$\bullet$};
\node(2) at (0,0) {$\bullet$};

\draw[->] (1) to [out=30, in=150] (2);
\draw[->] (1) to [out=-30, in=-150] (2);
\end{tikzpicture} \qquad
\begin{tikzpicture}[x=11pt,y=11pt,thick]\pgfsetlinewidth{0.5pt}
\node(1) at (0,0) {$\bullet$};

\draw[->] (1) to [out=150, in=30, loop] (1);
\draw[->] (1) to [out=-30, in=-150, loop] (1);
\end{tikzpicture}
\end{center}
As a conclusion, a groupoid is an equivalence relation one, if and only if, its has no parallel arrows.
\end{Remark}

\subsection{Bisets, two sided translation groupoid and the tensor product}\label{ssec:biset}

Let $\cG$ and $\cH$ be two groupoids and $(X, \vartheta,\varsigma)$ a triple consisting of a set $X$ and two maps $\varsigma \colon X \to \Go$, $\vartheta\colon X \to \Ho$. The following definitions are abstract formulations of those given in~\cite{Jelenc:2013, Moedijk/Mrcun:2005} for topological and Lie groupoids, see also \cite{ElKaoutit:2015, Kaoutit/Spinosa:2018}.

\begin{Definition}\label{def:biset}The triple $(X,\vartheta,\varsigma)$ is said to be an \emph{$(\cH,\cG)$-biset} (or \emph{groupoid-bisets}) if there is a left $\cH$-action $\lambda\colon \Ha \due \times {\Sscript{{s}}} { \Sscript{\vartheta}} X \to X$ and right $\cG$-action $\rho\colon X \due \times {\Sscript{\varsigma}} { \Sscript{{t}}} \Ga \to X$ such that
\begin{enumerate}\itemsep=0pt\samepage
\item For any $x \in X$, $h \in \cH_{\Sscript{1}}$, $g \in \cG_{\Sscript{1}}$ with $\vartheta(x)={s}(h)$ and $\varsigma(x)={t}(g)$, we have
\begin{gather*} \vartheta(xg) =\vartheta(x)\qquad \text{and}\qquad \varsigma(hx)=\varsigma(x).\end{gather*}
\item For any $ x \in X$, $h \in \cH_{\Sscript{1}}$ and $ g \in \cG_{\Sscript{1}}$ with $\varsigma(x)={t}(g)$, $\vartheta(x)={s}(h)$, we have $h(xg) = (hx)g$.
\end{enumerate}
\end{Definition}

In analogy with that was mentioned in Remark~\ref{rem:GrpdSets}, groupoids-bisets can be also realized as functors from the Cartesian product of groupoids to the core category of sets. Thus, the category of groupoid-bisets is isomorphic (as a symmetric monoidal category) to the category of (right) groupoid-sets over the Cartesian product groupoid, see \cite[Proposition~3.12]{Kaoutit/Spinosa:2018}.

\emph{The two sided translation groupoid} associated to a given $(\cH, \cG)$-biset $(X,\varsigma, \vartheta)$ is defined to be the groupoid $\cH \lJoin X \rJoin \cG$ whose set of objects is $X$ and set of arrows is
\begin{gather*}
\cH_{\Sscript{1}} \due \times {\Sscript{{s}}} {\Sscript{\vartheta}} X \due \times {\Sscript{\varsigma}}{\Sscript{{s}}} \cG_{\Sscript{1}} = \big\{ (h,x,g) \in\cH_{\Sscript{1}}\times X \times \cG_{\Sscript{1}}\,|\, {s}(h)= \vartheta(x), {s}(g)=\varsigma(x) \big\}.
\end{gather*}
The structure maps are
\begin{gather*}
{s}(h,x,g)=x,\qquad {t}(h,x,g)=hxg^{-1}\qquad \text{and}\qquad \iota_{\Sscript{x}}=(\iota_{\Sscript{\vartheta(x)}}, x,\iota_{\Sscript{\varsigma(x)}}).
\end{gather*}
The multiplication and the inverse are given by:
\begin{gather*}
(h,x,g) (h',x',g') = (hh',x',gg'),\qquad (h,x,g)^{-1}=\big(h^{-1}, hxg^{-1}, g^{-1}\big).
\end{gather*}

The \emph{orbit space of} $X$, is the quotient set $X / (\cH,\cG)$ defined using the equivalence relation $x \sim x'$, if and only if, there exist $h \in \Ha$ and $g \in \Ga$ with $s(h)=\vartheta(x)$ and $t(g)=\varsigma(x')$, such that $hx = x'g$. Thus it is the set of connected components of the associated two translation groupoid.

\begin{Example}\label{exam:bisets}Let $\upphi\colon \cH \to \cG$ be a morphism of groupoids. Consider, as in Example \ref{exam: HG}, the associated triples $(\Ho \due \times {\Sscript{\upphi_0}} { \Sscript{t}} \Ga, \vartheta, {\rm pr}_{\Sscript{1}},)$ and $(\Ga \due \times {\Sscript{s}} { \Sscript{\upphi_0}} \Ho, {\rm pr}_{\Sscript{2}}, \varsigma)$. These are $(\cH, \cG)$-biset and $(\cG,\cH)$-biset,, respectively.
\end{Example}

Next we recall the definition of the \emph{tensor product of two groupoid-bisets}, see for instance~\cite{ElKaoutit:2015, Kaoutit/Spinosa:2018} or~\cite{Kaoutit/Kowalzig:14}. Fix three groupoids $\cG$, $\cH$ and $\cK$. Given $(Y, \varkappa,\varrho)$ and $(X, \vartheta,\varsigma)$, a~$(\cG, \cH)$-biset and $(\cH,\cK)$-biset, respectively. Consider the map $\upomega: Y \due \times {\Sscript{\varrho}}{ \Sscript{\vartheta}} X \to \Ho$ sending $(y,x) \mapsto \varrho(y)=\vartheta(x)$. Then the pair $\big(Y \due \times {\Sscript{\varrho}}{ \Sscript{\vartheta}} X, \upomega \big)$ admits a structure of right $\cH$-set with action
\begin{gather*}
\big(Y \due \times {\Sscript{\varrho}}{ \Sscript{\vartheta}} X\big) \due \times {\Sscript{\upomega}}{ \Sscript{t}} \Ha \longrightarrow \big(Y \due \times {\Sscript{\varrho}}{
\Sscript{\vartheta}} X\big),\qquad \big( \big((y,x), h\big) \longmapsto \big(yh,h^{-1}x\big) \big).
\end{gather*}
Following the notation and the terminology of \cite[Remark 2.12]{ElKaoutit:2015}, we denote by $\big(Y \due \times {\Sscript{\varrho}}{ \Sscript{\vartheta}} X\big) / \cH:= Y \tensor{\cH}X$ the orbit set of the right $\cH$-set $\big(Y \due \times {\Sscript{\varrho}}{ \Sscript{\vartheta}} X, \upomega\big)$. We refer to $Y\tensor{\cH}X$ as \emph{the tensor product over $\cH$} of $Y$ and $X$. It turns out that $Y\tensor{\cH}X$ admits a structure of $(\cG,\cK)$-biset whose structure maps are given as follows. First, denote by $y\tensor{\cH}x$ the equivalence class of an element $(y,x) \in Y \due \times {\Sscript{\varrho}}{ \Sscript{\vartheta}} X$. That is, we have $yh\tensor{\cH}x=y\tensor{\cH}hx$ for every $h \in \Ha$ with $\varrho(y)=t(h)=\vartheta(x)$. Second, one can easily check that, the maps
\begin{gather*}
\td{\varkappa}\colon \ Y \tensor{\cH} X \to \Go, \quad \big( y\tensor{\cH}x \longmapsto \varkappa(y) \big) ;\qquad \td{\varsigma}\colon \ Y \tensor{\cH} X \to \Ko, \quad \big( y\tensor{\cH}x \longmapsto \varsigma(x) \big)
\end{gather*}
are well defined, in such a way that the following ones
\begin{gather*}
\big(Y \tensor{\cH} X\big) \due \times {\Sscript{\td{\varkappa}}}{ \Sscript{t} } \Ka \longrightarrow Y \tensor{\cH} X, \qquad \big( \big( y\tensor{\cH}x, k\big) \longrightarrow y\tensor{\cH}xk \big), \\
\Ga \due \times {\Sscript{s}}{ \Sscript{\td{\varkappa}} } \big(Y \tensor{\cH} X\big) \longrightarrow \big(Y \tensor{\cH}X\big), \qquad \big( \big(g, y\tensor{\cH}x\big) \longrightarrow gy\tensor{\cH}x \big)
\end{gather*}
define a structure of $(\cG,\cK)$-biset on $Y\tensor{\cH}X$, as claimed.

\subsection{Normal subgroupoids and quotients}\label{ssec:normal}
Given a morphism of groupoids $\upphi\colon \cH \to \cG$, we define \emph{the kernel of $\upphi$} and denote by $\Ker{\upphi}$ (or by $\upphi^{\Sscript{k}}\colon \Ker{\upphi} \hookrightarrow \cH$), the groupoid whose underlying category is a subcategory of $\cH$ given by following pair of sets:
\begin{gather*}
\Ker{\upphi}_{\Sscript{0}}=\Ho,\qquad \Ker{\upphi}_{\Sscript{1}}=\big\{ h \in \Ha\,|\, \phia(h)=\iota_{\Sscript{ \upphi_0( s(h))}}= \iota_{\Sscript{\upphi_0( t(h))}}\big\}.
\end{gather*}

In other words, $\Ker{\upphi}$ is the subcategory of $\cH$ whose arrows are mapped to identities by~$\upphi$. In particular, the isotropy groups of $\Ker{\upphi}$ coincide with the kernels of the isotropy maps. Thus, we have that
\begin{gather*}
\Ker{\upphi}^{\Sscript{u}}= {\rm Ker}\big(\upphi^{\Sscript{u}}\colon \cH^{\Sscript{u}} \to \cG^{\Sscript{\upphi_{\Sscript{0}}(u)}} \big),\qquad \text{ for any object } u \in \Ho.
\end{gather*}
Furthermore, for any arrow $h \in \Ha$, we have that
\begin{gather*}
\Ad{h}\big( \Ker{\upphi}^{\Sscript{s(h)}} \big) = \Ker{\upphi}^{\Sscript{t(h)}},
\end{gather*}
where $\Ad{h}$ is the adjoint operator of $h$ defined in \eqref{Eq:Ad}. These properties motivate the following definition.
\begin{Definition}\label{def:Normal}Let $\cH$ be a groupoid. A \emph{normal subgroupoid} of $\cH$ is a subcategory $\cN \hookrightarrow \cH$ such that
\begin{enumerate}[(i)]
\item $\cN_{\Sscript{0}} =\cH_{\Sscript{0}}$;
\item For every $h \in \Ha$, we have $\Ad{h}\big(\cN^{\Sscript{s(h)}}\big) = \cN^{\Sscript{t(h)}} $ as subgroups of $\cH^{\Sscript{t(h)}}$.
\end{enumerate}
\end{Definition}
Notice that given a normal subgroupoid $\cN$ of $\cH$, then each of the isotropy groups $\cN^{\Sscript{u}}$, $u \in \No$, is a normal subgroup of $\cH^{\Sscript{u}}$. In particular the isotropy groupoid $\cN^{\Sscript{(i)}}$ of $\cN$, as defined in Example~\ref{exam:IsotropyGrpd}, is a normal subgroupoid of the isotropy groupoid $\cH^{\Sscript{(i)}}$ of $\cH$.
\begin{Example}\label{exam:Normal}
As we have seen above the kernel of any morphism of groupoids is a normal subgroupoid. The converse also holds true (see Proposition~\ref{prop:quotient} below). On the other hand, if $H \lhd G$ is a normal subgroup, then $(X \times H \times X, X)$ is clearly a normal subgroupoid of the induced groupoid $(X \times G \times X, X)$, see Example~\ref{exam:induced}. Now taking $\cR$ any equivalence relation on a set $X$, and consider the associated groupoid as in Example~\ref{exam:X}. Then $(\cR,X)$ is a normal subgroupoid of the groupoid of pairs $(X\times X, X)$.
\end{Example}

Next we recall the construction of the quotient groupoid from a given normal subgroupoid. Let $\cN$ be a normal subgroupoid of~$\cH$. Clearly $(\Ha,s)$ and~$(\Ha,t)$ are, respectively, right $\cN$-set and left $\cN$-set with actions given by the multiplication of $\cH$:
\begin{gather*}
\Ha \due \times {\Sscript{s}} { \Sscript{t}} \cN_{\Sscript{1}} \longrightarrow \Ha,\qquad \big((h,e) \longmapsto
he \big) \qquad \Na \due \times {\Sscript{s}} { \Sscript{t}} \cH_{\Sscript{1}} \longrightarrow \Ha,\qquad \big((e,h) \longmapsto eh \big).
\end{gather*}
Therefore $(\Ha, s, t)$ is a $\cN$-biset in the sense of Section~\ref{ssec:biset}. We denote its orbit set by~$\Ha/\cN$. That is, the quotient set of $\Ha$ modulo the equivalence relation $h \sim h' \Leftrightarrow \exists\, (e,h,e') \in \cN_{\Sscript{1}} \due \times {\Sscript{s}} { \Sscript{t}} \Ha $ $\due \times {\Sscript{s}}{ \Sscript{s}} \cN_{\Sscript{1}}$ such that $eh = h'e'$. On the other hand, we can consider the quotient set of $\Ho$ modulo the relation: $u \sim u' \Leftrightarrow \exists\, e \in \cN_{\Sscript{1}}$ such that $s(e)=u$ and $t(e)=u'$. Denotes by $\cH_{\Sscript{0}}/\cN$ the associated quotient set and by $\cH/\cN:=\big(\cH_{\Sscript{1}}/\cN, \cH_{\Sscript{0}}/\cN \big)$ the pair of sets, which going to be \emph{the quotient groupoid}.

\begin{Proposition}\label{prop:quotient}Let $\cN$ be a normal subgroupoid of~$\cH$. Then the pair of orbit sets $\cH/\cN$ admits a structure of groupoid such that there is a ``sequence'' of groupoids:
\begin{gather*}
\xymatrix{\cN \ar@{^{(}->}[r] & \cH \ar@{->>}[r] & \cH/\cN.}
\end{gather*}
Furthermore, any morphism of groupoids $\upphi\colon \cH \to \cG$ with $\cN \subseteq \Ker{\upphi} $, factors uniquely as
\begin{gather*}
\xymatrix{ \cH \ar@{->}^-{\upphi}[r] \ar@{->>}_-{\uppi}[rd] & \cG
\\ & \cH/\cN. \ar@{-->}_-{\bara{\phi}}[u]. }
\end{gather*}
\end{Proposition}
\begin{proof}The source, the target and the identity maps of $\cH/\cN$, are defined using those of $\cH$, that is, for a given arrow $\bara{h} \in \cH_{\Sscript{1}}/\cN$, we set $s\big(\bara{h}\big)=\bara{s(h)}$, $t(\bara{h})=\bara{t(h)}$ and $\iota_{\bara {\Sscript{u}}}=\bara{\iota_{\Sscript{u}}}$, for any $\bara{u} \in \Ho/\cN$. These are well defined maps, since they are independent form the chosen representative of the equivalence class. The multiplication is defined by
\begin{gather*}
\big(\cH/\cN\big) \due \times { \Sscript{\bara{s}}} {\Sscript{\bara{t}}} \big(\cH/\cN\big) \longrightarrow\big(\cH/\cN\big), \qquad \big( (\bara{h}, \bara{h'}) \longmapsto \bara{hh'} \big)
\end{gather*}
this is a well defined associative multiplication thanks to condition (ii) of Definition~\ref{def:Normal}. Lastly, the inverse of an arrow $\bara{h} \in \cH/\cN$ is given by the class of the inverse $\bara{h^{-1}}$. The canonical map $(h,u) \mapsto (\bara{h}, \bara{u})$ defines morphism of groupoids $\cH \to \cH/\cN$ whose kernel is $\cN \hookrightarrow \cH$. The proof of the rest of the statements is immediate.
\end{proof}

The fact that normal subgroups can be characterize as the invariant subgroups under the conjugation action, can be immediately extended to the groupoids context, as the following Lemma shows. But first let us observe that the conjugation operation of equation \eqref{Eq:Ad}, induces a left $\cH$-action on the set of objects of the isotropy groupoid $\Hi{}_{\Sscript{1}}$ with the structure map $\varsigma\colon \Hi{}_{\Sscript{1}}=\cup_{u \in \Ho}\cH^{\Sscript{u}} \to \Ho$ (source or the target of the isotropy groupoid $\Hi$). That is,
\begin{gather}\label{Eq:Gi}
\Ha \due \times {\Sscript{s}} { \Sscript{\varsigma}}\Hi{}_{\Sscript{1}} \longrightarrow \Hi{}_{\Sscript{1}},\qquad \big((h,l) \longmapsto hlh^{-1} \big)
\end{gather}
defines a left $\cH$-action on $\Hi{}_{\Sscript{1}}$.
\begin{Lemma}\label{lema:normalsugpd}Let $\cH$ be a groupoid and $\cN \hookrightarrow \cH$ a subcategory with $\No=\Ho$. Then $\cN$ is a~normal subgroupoids if and only if $\Ni{}_{\Sscript{1}}$ is an $\cH$-invariant subset of $\Hi{}_{\Sscript{1}}$ with respect to the action of equation~\eqref{Eq:Gi}.
\end{Lemma}
\begin{proof}Straightforward.
\end{proof}

\section{Linear representations of groupoid. Revisited}\label{sec:rep}
We provide in this section the construction and the basic properties of the induction, restriction and co-induction functors attached to a morphism of groupoids, and connect the categories of linear representations. These properties are essential to follow the arguments presented in the forthcoming sections. The material presented here is probably well known to specialists, with the exception perhaps the result dealing with the characterization of linear representations of quotient groupoid that has its own interest. Nevertheless, we have preferred to give a self-contained and elementary exposition, which we think is accessible to wide range of the audience.

\subsection{Linear representations: basic properties}\label{ssec:Rep}
Given a groupoid $\cG$, we denote, as in Section~\ref{ssec:raroes}, by $\Rep{\cG}$ the category of all $\Bbbk$-linear $\cG$-representations. The $\Bbbk$-vector space of morphisms between two $\cG$-representations $\cV$ and $\cV'$, will be denoted by $\hom{\cG}{\cV}{\cV'}$.

To any representation $\cV$ one associated the functor $\tilde{\cV}\colon \cG \to \Sets$ by forgetting the $\Bbbk$-vector space structure of the representation. The same notation $\tilde{\tF}$ will be used for any morphism $\tF \in \hom{\cG}{\cV}{\cV'}$. The resulting functor $\td{(-)}\colon \Rep{\cG}=[\cG, \Vect] \to [\cG, \Sets]$ is called \emph{the forgetful functor}. The image of an object $x \in \Go$ by the representation $\cV$, is denoted by $\Vx=\cV(x)$. Given an arrow $g \in \Ga$, we denote by $\Vg\colon \cV_{\Sscript{s(g)}} \to \cV_{\Sscript{t(g)}}$ the image of~$g$ by~$\cV$.

\begin{Remark}\label{rem:FrameGrpd}As in the case of groups, a linear representation can be defined via a morphism of groupoids, see also \cite[Definition~1.11]{Renault:1980}. Namely, fix a groupoid $\cG$ and recall (see for instance \cite[p.~98]{Sternberg:1994}) that a ``\emph{vector bundle}'' or a $\Vect$-\emph{bundle} over $\Go$, is a disjoint union $E=\biguplus_{ x \in \Go}E_{\Sscript{x}}$ of $\Bbbk$-vector spaces with the canonical projection $\pi\colon E \to \Go$\footnote{This is a vector bundle (possibly with infinite dimensional fibers) in the topological sense \cite[Definition~2.1]{Karoubi:1978}, by taking the discrete topology on both sets $\Go$ and~$\Bbbk$.}. We denote a vector bundle over $\Go$ simply by $(E,\pi)$ and call the vector space $E_{\Sscript{x}}$ \emph{the fibre of $E$ at $x$}. In this way, the frame groupoid ${\rm Iso}(E,\pi)$ of $(E,\pi)$, defined in Example~\ref{exam:frame}, has $\Go$ as a set of objects, and the set of arrows from~$x$ to~$x'$ is determined by~${\rm Iso}(E,\pi)(x,x'):={\rm Iso}_{\Sscript{\Bbbk}}( E_{\Sscript{x}}, E_{\Sscript{x'}})$ the set of all $\Bbbk$-linear isomorphisms from $E_{\Sscript{x}}$ to $E_{\Sscript{x'}}$. In the same direction, any morphism of groupoids $\upmu\colon \cG \to {\rm Iso}(E,\pi)$ whose objects map is the identity $\upmu_{\Sscript{0}}={\rm id}_{\Sscript{\Go}}$, gives rise to a linear representation of $\cG$. Namely, the corresponding representation is given by the functor $\cE\colon \cG \to \Vect$ acting on objects by $x \mapsto E_{\Sscript{x}}$ (the fibre of $E$ at $x$) and on arrows by $g \mapsto \big[\cE^{\Sscript{g}}=\upmu_{\Sscript{1}}(g)\colon E_{\Sscript{s(g)}} \to E_{\Sscript{t(g)}}\big]$.

Conversely, assume we are given a $\cG$-representation $\cV$. Then the pair $(\bara{\cV}, \pi_{\Sscript{\cV}})$ which consists of the disjoint union $\bara{\cV}=\biguplus_{x \in \Go}\Vx$ and the projection $\pi_{\Sscript{\cV}}\colon \bar{\cV} \to \Go$, clearly defines a vector bundle over $\Go$. Moreover, we have a morphism of groupoids, defined by
\begin{gather}\label{Eq:rhoV}
\varrho_{\Sscript{\cV}}:=\upnu\colon \ \cG \longrightarrow {\rm Iso}(\bara{\cV}, \pi_{\Sscript{\cV}}), \qquad \upnu_{\Sscript{0}}(x)=x, \qquad \text{and} \qquad \upnu_{\Sscript{1}}(g)=\Vg,
\end{gather}
for every $x \in \Go$ and $g \in \Ga$.

On the other hand it is clear that for any $\cG$-representation $\cV$, the pair $(\bara{\cV}, \pi_{\Sscript{\cV}})$ with the map
\begin{gather}\label{Eq:actionV}
\Ga \due \times {\Sscript{s}} { \Sscript{\pi_{\cV}}} \bara{\cV} \longrightarrow \bara{\cV},\qquad \big((g,v) \longmapsto gv:=\Vg(v) \big)
\end{gather}
lead to a left $\cG$-set structure on the bundle $\bara{\cV}$, in the sense of the left version of Definition~\ref{def:Gset}. This in fact establishes a faithful functor from the category of $\cG$-representations to the category of left $\cG$-sets, which is in turn the composition of the forgetful functor $\td{(-)}$ and the functor discussed in the Remark~\ref{rem:GrpdSets}, that goes from the category of functors $[\cG, \Sets]$ to the category of left $\cG$-sets.
\end{Remark}

It is well known, see for instance~\cite{Mitchell:1965}, that the category $\Rep{\cG}$ is an abelian symmetric monoidal category with a set of small generators. The monoidal structure is extracted from that of $\Vect$, that is, for any two representations $\cU$ and $\cV$, their tensor product is the functor $\cU \tensor{}\cV\colon \cG \to \Vect$ defined by $(\cU\tensor{}\cV)_{\Sscript{x}}= \cU_{\Sscript{x}} \tensor{\Bbbk}\cV_{\Sscript{x}}$ and
$(\cU\tensor{}\cV)^{\Sscript{g}}= \cU^{\Sscript{g}}\tensor{\Bbbk} \cV^{\Sscript{g}}$, for every $x \in \Go$ and $g \in \Ga$.

The category $\Rep{\cG}$, is in particular locally small, in the sense that the class of subobjects of any object is actually a set. The zero representation well be denoted by $\bf 0$ and the identity representation (with respect to the tensor product), or the trivial representation, by~$\bf 1$. Moreover, to any representation one can associate its dual representation. Indeed, take a representation~$\cV$, for any object $x \in \Go$, set $(\cV^*)_{\Sscript{x}}:= (\cV_{\Sscript{x}})^*=\hom{\Bbbk}{\cV_{\Sscript{x}}}{\Bbbk}$ the linear dual of the $\Bbbk$-vector space~$\cV_{\Sscript{x}}$, and set $(\cV^*){}^g:= \big(\cV^{\Sscript{g^{-1}}}\big)^*$ for a given arrow~$g \in \Ga$. In this way, we obtain a representation~$\cV^*$ with a canonical morphism of representations $\cV^* \tensor{} \cV \to {\bf 1}$ fibrewise given by the evaluation maps $\cV^*_{\Sscript{x}} \tensor{\Bbbk} \Vx \to \Bbbk$ acting by $\varphi \tensor{\Bbbk}v \mapsto \varphi(v)$, for every $x \in \Go$.

We say that \emph{a representation $\cV \in \Rep{\cG}$ is finite}, when its image lands in the subcategory $\vect$ of finite dimensional $\Bbbk$-vector spaces. The full subcategory of finite representations is then an abelian symmetric rigid monoidal category.

\begin{Example}\label{exam:One}For instance a finite representation where each one of its fibres is a one-dimen\-sio\-nal $\Bbbk$-vector space can be identified with a~family of elements $\{ \lambda_{\Sscript{(s(g), t(g))}}\}_{\Sscript{ g \in \Ga}}$ in $\Bbbk^{\Sscript{\times}}$, the multiplicative group of~$\Bbbk$,
satisfying
\begin{gather*}
\lambda_{(s(g), t(g))} \lambda_{(s(h), t(h))} = \lambda_{(s(hg), t(hg))}, \qquad \text{whenever} \qquad t(g)=s(h), \qquad \text{and} \\ \lambda_{(x, x)} = 1_{\Sscript{\Bbbk}}, \qquad \text{for every} \quad x \in \Go.
\end{gather*}

In the second condition, the term $\lambda{\Sscript{(x,x)}}$ stands for the $\iota_{\Sscript{x}}$'s projection of $\lambda$. The family \linebreak $\{ \lambda_{\Sscript{(s(g), t(g))}}\}_{\Sscript{ g \in \Ga}}$ in $\Bbbk^{\Sscript{\times}}$, where $\lambda_{\Sscript{(s(g), t(g))}}=1_{\Sscript{\Bbbk}}$, for every $g \in \Ga$, corresponds then to the trivial representation $\bf 1$.
\end{Example}

To any representation $\cV \in \Rep{\cG}$, one can consider the projective and the inductive limits of its underlying functor, since this one lands in the Grothendieck category~$\Vect$. These $\Bbbk$-vector spaces, are denoted by $\prolimit{\cG}{\cV}$ and $\injlimit{\cG}{\cV}$, respectively.

Given a representation $\cV$ we can define as follows its $\cG$-invariant subrepresentation. For any $x \in \Go$, we set~$\Vgx$ the subspace of $\Vx$ invariant under the action of the isotropy group $\cG^{\Sscript{x}}$. That is,
\begin{gather*}
\Vgx=\big\{ v \in \Vx\,|\, \cV^{\Sscript{l}}(v):= l v=v, \text{ for all } l \in \cG^{\Sscript{x}} \big\}.
\end{gather*}
Now, take an arrow $g \in \Ga$, and a vector $v \in \Vgsg$. Then, for any $q \in \cG^{\Sscript{t(g)}}$, we have that
\begin{gather*}
q ( g v) = g \big( \big(g^{-1}q g\big) v \big) = g v.
\end{gather*}
Therefore, the image under the linear map $\cV^{\Sscript{g}}$ of any vector in $ \Vgsg$ lands in $ \Vgtg$. The same holds true interchanging $g$ by $g^{-1}$. This means that for any arrow $g \in
\Ga$, we have a commutative diagram
\begin{gather*}
\xymatrix{ \cV_{\Sscript{s(g)}}\ar@{->}^-{\Sscript{\cV^{\Sscript{g}}}}[rr] && \cV_{\Sscript{t(g)}}\\ \cV^{\Sscript{\cG}}_{\Sscript{s(g)}} \ar@{-->}^-{}[rr]
\ar@{^(->}^-{}[u] & & \cV^{\Sscript{\cG}}{}_{\Sscript{t(g)}}.\ar@{^(->}^-{}[u] }
\end{gather*}
In this way we obtain a representation $\cV^{\Sscript{\cG}}\colon \cG \to \Vect$ with a monomorphism $\cV^{\Sscript{\cG}} \hookrightarrow\cV$ in $\Rep{\cG}$. This representation is referred to as the \emph{$\cG$-invariant subrepresentation of $\cV$}.

\begin{Remark}\label{rem:invaraiant} If $\cG$ is a groupoid with only one object, that is a group, then $\cV^{\Sscript{\cG}} = \hom{\cG}{\bf 1}{\cV}$. In general, however, we can not directly relate this later vector space with the fibres of the $\cG$-invariant representation. More precisely, we have that $\prolimit{\cG}{\cV} = \hom{\cG}{\bf 1}{\cV}$ as vector spaces, and the following commutative diagram of vector spaces:
\begin{gather*}
\xymatrix@C=45pt{ 0 \ar@{->}^-{}[r] & \prolimit{\cG}{\cV} \ar@{->}^-{}[r] \ar@{-->}^-{}[dr] & \underset{x \in \Go}{\prod} \Vx \ar@{->}^-{\pi}[r] & \underset{g \in \Ga}{\prod} \cV_{\Sscript{t(g)}} \\ 0 \ar@{->}^-{}[r] & \prolimit{\cG}{\cV^{\Sscript{\cG}}} \ar@{->}^-{}[r] \ar@{->}^-{}[u] & \underset{x \in \Go}{\prod} \Vgx \ar@{->}^-{}[u] \ar@{->}^-{}[r] & \underset{g \in \Ga}{\prod} \cV_{\Sscript{t(g)}}, \ar@{->}^-{}[u] \\ & 0 \ar@{->}^-{}[u] & 0 \ar@{->}^-{}[u] &}
\end{gather*}
where for every $g \in \Ga$, we have $\pi{\Sscript{g}}= \Vg \circ p_{\Sscript{s(g)}} - p_{\Sscript{t(g)}}$ and the $p$'s are the canonical projections. Then the dashed monomorphism of vector spaces is not necessarily an isomorphism.
\end{Remark}

In a dual way, one defines the \emph{co-invariant representation}. Specifically, let $\cV$ be a $\cG$-representation, for any $x \in \Go$, we set the quotient $\Bbbk$-vector space
\begin{gather*}
\cV{\Sscript{\Gx}}: = \Vx/ {\rm Span}_{\Sscript{\Bbbk}} \big\{ ev-v\,|\, e \in \Gx, v \in \Vx \big\}.
\end{gather*}
Given now an arrow $g \in\Ga$, we have that the linear map $\cV^{\Sscript{g}}$ extend to the quotients, that is, we have a~commutative diagram
\begin{gather*}
\xymatrix{ \cV_{\Sscript{s(g)}}
\ar@{->}^-{\Sscript{\cV^{\Sscript{g}}}}[rr] \ar@{->>}^-{}[d] &&
\cV_{\Sscript{t(g)}} \ar@{->>}^-{}[d] \\ \cV_{\Sscript{\cG^{s(g)}}}
\ar@{-->}^-{\Sscript{\bara{\cV^{\Sscript{g}}}}}[rr] & &
\cV_{\Sscript{\cG^{t(g)}}}. }
\end{gather*}
Therefore, the family $\{(\cV_{\Sscript{\cG^{s(g)}}}, \bara{\cV^{\Sscript{g}}} )\}_{x \in, \Go}$ defines a representation which we denote by $\cV_{\Sscript{\cG}}$. In this way, we have a canonical epimorphism $\cV \twoheadrightarrow \cV_{\Sscript{\cG}}$ in the category $\Rep{\cG}$ given fibrewise by the linear map $\pi_{\Sscript{x}}\colon \Vx \twoheadrightarrow \cV_{\Sscript{\cG^x}}$. In analogy with group theory context, the representation $\cV_{\Sscript{\cG}}$ is referred to as \emph{the coinvariant quotient representation of $\cV$}.

\begin{Remark}\label{rem:coinv}Similar to Remark \ref{rem:invaraiant}, the coinvaraint representation $\cV_{\Sscript{\cG}}$ is related to the limit of the representation $\injlimit{\cG}{\cV}$ and also to the vector space
$\hom{\cG}{\cV}{\bf 1}$. More precisely, we have, form one hand, an isomorphism of $\Bbbk$-vector spaces
\begin{gather*}
\Big(\big(\injlimit{\cG}{\cV} \big)^* \longrightarrow
\hom{\cG}{\cV}{\bf 1}, \;\; \varphi \longmapsto \big( \varphi \circ
\zeta_{\Sscript{x}}\big)_{\Sscript{x \in \Go}} \Big); \\ \Big(
\hom{\cG}{\cV}{\bf 1} \longrightarrow \big(\injlimit{\cG}{\cV}
\big)^*,\;\;\injlimit{\cG}{f_{\Sscript{x}}} \longmapsto \big(
f_{\Sscript{x}}\big)_{\Sscript{x \in \Go}}\Big),
\end{gather*}
where $\{\zeta_{\Sscript{x}}\colon \cV_{\Sscript{x}} \to \injlimit{\cG}{\cV}\}_{\Sscript{x \in \Go}}$ are the structural maps of the stated limit. On the other hand, since each of the maps $\zeta_{\Sscript{x}}$ factors through the quotient $\cV_{\Sscript{\cG^x}}$, we have a family of maps $\bara{\zeta_{\Sscript{x}}}\colon \cV_{\Sscript{\cG^x}} \to \injlimit{\cG}{\cV}$ whose direct sum rends the following diagram
\begin{gather*}
\xymatrix@C=45pt{ \bigoplus_{g \in \Ga} \cV_{\Sscript{s(g)}}
\ar@{->}^-{\tau}[r] \ar@{->}^-{}[d] & \bigoplus_{x \in \Go}
\cV_{\Sscript{x}} \ar@{->}^-{\tau^{\Sscript{c}}}[r] \ar@{->}^-{}[d]
& \injlimit{\cG}{\cV} \ar@{->}^-{}[r] \ar@{->}^-{}[d] & 0 \\
\bigoplus_{g \in \Ga} \cV_{\Sscript{\cG^{s(g)}}}
\ar@{->}^-{}[r] & \bigoplus_{x \in \Go} \cV_{\Sscript{\cG^x}}
\ar@{->}^-{}[d] \ar@{-->}_-{\oplus_{\Sscript{x \in
\Go}}\bara{\zeta_{\Sscript{x}}} }[ru] \ar@{->}^-{}[r] &
\injlimit{\cG}{\cV_{\Sscript{\cG}}} \ar@{->}^-{}[d] \ar@{->}^-{}[r]
& 0 \\ & 0 & 0 & }
\end{gather*}
commutative, where $\tau$ is given by $\tau_{\Sscript{g}}=
\tau_{\Sscript{t(g)}} \circ \cV^{\Sscript{g}}-
\tau_{\Sscript{s(g)}}$, for every $g \in \Ga$. The dashed
epimorphism is then not necessarily an isomorphism.
\end{Remark}

We finish this subsection by the following observation.

\begin{Lemma}\label{lema:Hom}Let $\cG$ be a groupoid and $\cV$, $\cU$ two representations in $\Rep{\cG}$. Then the fa\-mily of $\Bbbk$-vector spaces $\big\{\hom{\Bbbk}{\Ux}{\Vx}\big\}_{x \in \Go }$ defines a~representation in $\Rep{\cG}$ denoted by $\hom{\Bbbk}{\cU}{\cV}$. In particular, if $\cU$ is a finite representation then
\begin{gather*}
\cV \tensor{} \cU^{*} \cong \hom{\Bbbk}{\cU}{\cV},
\end{gather*}
an isomorphism in the category $\Rep{\cG}$. Furthermore, for any $\cU$ and $\cV$, we have
\begin{gather*}
\prolimit{\cG}{\hom{\Bbbk}{\cU}{\cV}^{\Sscript{\cG}}} =\hom{\cG}{\cU}{\cV}.
\end{gather*}
\end{Lemma}
\begin{proof}The action of $\hom{\Bbbk}{\cU}{\cV}$ on a given arrow $g \in \Ga$, is defined by the following linear isomorphism
\begin{gather*}
\hom{\Bbbk}{\cU_{\Sscript{s(g)}}}{\cV_{\Sscript{s(g)}}}\longrightarrow \hom{\Bbbk}{\cU_{\Sscript{t(g)}}}{\cV_{\Sscript{t(g)}}}, \qquad \big( \sigma \longmapsto \Vg \circ \sigma \circ \Umg \big).
\end{gather*}
This clearly defines an object in $\Rep{\cG}$. If $\cU$ is a finite representation, then each fibre $\hom{\Bbbk}{\Ux}{\Vx}$ is linearly isomorphic to the $\Bbbk$-vector space $\Vx\tensor{\Bbbk}\Ux^*$. This family of linear isomorphisms leads in fact to the stated natural isomorphism.

For the proof of last statement, let us consider the following well defined map:
\begin{gather*}
\xymatrix@R=0pt{ \hom{\cG}{\cU}{\cV} \ar@{->}^-{\fk{z}_{\Sscript{x}} }[rr] & & \hom{\Bbbk}{\Ux}{\Vx}^{\Sscript{\cG^x}}, \\ \alpha \ar@{|->}^-{}[rr]
& & \alpha_{\Sscript{x}} }
\end{gather*}
for every $x \in \Go$. Each one of these maps is a $\Bbbk$-linear map where $\hom{\cG}{\cU}{\cV}$ is endowed with the structure of $\Bbbk$-vector space fibrewise inherited from that of $\cV$. This leads to a projective system which is in turn the universal one. Thus, $\hom{\cG}{\cU}{\cV}=\prolimit{\cG}{\hom{\Bbbk}{\cU}{\cV}^{\Sscript{\cG}}}$ as claimed.
\end{proof}

\subsection{The restriction functor}\label{ssec:Res}
Let $\upphi\colon \cH \to \cG$ be a morphism of groupoids. The \emph{restriction functor} is the functor defined by
\begin{gather}\label{Eq:restriction}
\phil\colon \ \Rep{\cG} \longrightarrow \Rep{\cH}, \qquad \big( \cV\longrightarrow \cV \circ \upphi;\ \tF \longrightarrow\tF_{\Sscript{\upphi}} \big),
\end{gather}
where the notation is the obvious one. In the subsequent we analyze the property of the restriction functor corresponding to a normal subgroupoid. Such a property is in fact a generalization of \cite[Proposition~2.3.2]{Kowalski:2014} and of course has its own interest in groupoids context.

\begin{Proposition}[representations of quotients]\label{prop:Res} Let $\cN$ be a normal subgroupoid of $\cH$ and denote by $\uppi\colon \cH \to \cG=\cH/\cN$ the canonical projection. Then the restriction functor induces an isomorphism of categories between $\Rep{\cG}$ and the full subcategory $\Rep{\cH}^{\Sscript{\cN}}$ of $\Rep{\cH}$ whose representations are trivial on $\cN$, that is, an object in $\Rep{\cH}^{\Sscript{\cN}}$ is a representation~$\cV$ of~$\cH$ such that $\cN \subset \Ker{\varrho_{\Sscript{\cV}}}$, where $\varrho_{\Sscript{\cV}}$ is as in equation~\eqref{Eq:rhoV}.
\end{Proposition}
\begin{proof}We know that $\uppi_{*}\colon \Rep{\cG} \!\to\! \Rep{\cH}$ has the image in the full subcategory $\Rep{\cH}^{\Sscript{\cN}}\!$. Let us denote also by $\uppi_{*}\colon \Rep{\cG} \to \Rep{\cH}^{\Sscript{\cN}}$ the resulting functor. The inverse of this functor is construct with the help of Proposition~\ref{prop:quotient}. Explicitly, given a representation $(\cV, \varrho_{\Sscript{\cV}})$ in $\Rep{\cH}^{\Sscript{\cN}}$, we have by equation~\eqref{Eq:rhoV} a morphism of groupoids $\varrho_{\Sscript{\cV}}\colon \cH \to {\rm Iso}(\bara{\cV}, \uppi_{\Sscript{\cV}})$ with $\cN \subset \Ker{\varrho_{\Sscript{\cV}}}$. Then by Proposition~\ref{prop:quotient}, we have a~representation $\bara{\varrho_{\Sscript{\cV}}}\colon \cG=\cH/\cN \to {\rm Iso}(\bara{\cV}, \pi_{\Sscript{\cV}})$. This establishes a functor $\uppi^{*}\colon \Rep{\cH}^{\Sscript{\cN}} \to
\Rep{\cG}$ which turns out to be the inverse of~$\uppi_{*}$.
\end{proof}

\begin{Remark}\label{rem:WN} Let $\cN$ be a normal subgroupoid of $\cH$. Then for any representation $\cW \in \Rep{\cH}$, we can consider the assignment
\begin{gather*}
\cW^{\Sscript{\cN}}\colon \ \Ho \longrightarrow \Vect, \qquad \big( u \longmapsto \cW^{\Sscript{\cN^u}} \big),
\end{gather*}
where the subspace $\cW^{\Sscript{\cN^u}}$ of $\Wu$ consists of those vectors which are invariant under the action of loops in $\cN$, that is, those $v \in \Wu$ such that $\We(v)=e v=v$, for every $e \in \cN^{\Sscript{u}}$. It turns out that $\cW^{\Sscript{\cN}}$ gives a well defined functor which acts by restriction on arrows, because $\cN$ is normal. Indeed, take a~vector $w \in \cW^{\Sscript{\cN^{s(h)}}}$ for some $h \in \Ha$ and $e \in \cN^{\Sscript{t(h)}}$. Then, we get that $e (h w)= h \big(\big(h^{-1}eh\big)w \big)=h w$. Therefore, $\cW^{\Sscript{\cN}}$ is a sub-representation of the $\cH$-representation~$\cW$.
\end{Remark}

\subsection{The induction functor}\label{ssec:Ind}
Let $\upphi: \cH \to \cG$ be a morphism of groupoids. Fix an object $x \in \Go$ and consider the functor
\begin{gather*}
\phix\colon \ \cH \longrightarrow \Sets, \qquad \big( u \longmapsto \cG(x,\upphi_{\Sscript{0}}(u)) \big),
\end{gather*}
which acts as follows: For any arrow $h\in \Ha$, we have $\phix(h)=\cG(x, \upphi_{\Sscript{1}}(h))$:
\begin{gather}\label{Eq:Tartu}
\phix(h) \colon \ \cG\big(x, \upphi_{\Sscript{0}}(s(h)) \big) \longrightarrow \cG\big(x, \upphi_{\Sscript{0}}(t(h)) \big), \qquad \big( p \longmapsto \upphi_{\Sscript{1}}(h) p\big).
\end{gather}

As a contravariant functor $\upphi^{-}\colon \cG \to \left[\cH, \Sets\right]$, it acts by
\begin{gather}\label{Eq:Oslo}
\upphi^{\Sscript{g}}_u=\cG(g,\phi_{\Sscript{0}}(u))\colon \ \phix(u)=\cG(x,\upphi_{\Sscript{0}}(u) ) \longrightarrow \phixp(u)=\cG(x',\upphi_{\Sscript{0}}(u)), \qquad \big( q \longmapsto q g \big),
\end{gather}
for every object $u \in \Ho$ and arrow $g\colon x' \to x$ in $\Ga$.

\looseness=-1 Given a representation $\cW$ in $\Rep{\cH}$, we can construct a family of $\Bbbk$-vector spaces \linebreak $ \big\{ {\rm Nat}\big(\phix, \tilde{\cW}\big) \big\}_{x \in \Go}$ by using the fibrewise $\Bbbk$-vector space structure of $\cW$, where $\tilde{\cW}\colon \cH \to \Sets$ is the image of $\cW$ by the forgetful functor. More precisely, for each $x \in \Go$, the set of natural transformations ${\rm Nat}\big(\phix,\tilde{\cW}\big)$ admits a canonical structure of $\Bbbk$-vector space given componentwise by
\begin{gather*}
(\alpha+ \beta)_{u} =\alpha_u + \beta_u,\ (\lambda \alpha)_{u}= \lambda \alpha_u \colon \ \phix_{\Sscript{u}} \longrightarrow \Wu,
\end{gather*}
for every $u \in \Ho$ and $\alpha, \beta \in {\rm Nat}\big(\phix, \tilde{\cW}\big)$. We have then construct a functor
\begin{gather}\label{Eq:Phistarup}
\begin{gathered}
\xymatrix@R=0pt{ \upphi^*(\cW) \colon \ \cG \ar@{->}^-{}[rr] & & \Vect, \\ x \ar@{|->}^-{}[rr] & & {\rm Nat}\big(\phix, \tilde{\cW}\big), \\ g \ar@{|->}[rr] & & {\rm Nat}\big(\upphi^{\Sscript{g}}, \tilde{\cW}\big). }
\end{gathered}
\end{gather}

\begin{Lemma}\label{lema:Induction}Let $\upphi\colon \cH \to \cG$ be a morphism of groupoids. Then the assignment $\upphi^{*}\colon \Rep{\cH} \to \Rep{\cG}$ given in equation~\eqref{Eq:Phistarup}, gives a well defined functor referred to as {\rm the induction functor}.
\end{Lemma}
\begin{proof}Given a morphism $\tF\colon \cW \to \cW'$ in the category $\Rep{\cH}$ and an arrow $g \in \Ga$, we need to check that the following diagram
\begin{gather*}
\xymatrix@C=45pt@R=35pt{ {\rm Nat}\big(\upphi^{\Sscript{s(g)}},
\tilde{\cW}\big) \ar@{->}^-{ \upphi^*(\tF)_{\Sscript{s(g)}} = {\rm
Nat}(\upphi^{\Sscript{s(g)}},\tilde{\tF})}[rr] \ar@{->}|-{{\rm
Nat}\big(\upphi^{\Sscript{g}}, \tilde{\cW}\big)}[d] & & {\rm
Nat}(\upphi^{\Sscript{s(g)}}, \tilde{\cW'}) \ar@{->}|-{{\rm
Nat}(\upphi^{\Sscript{g}}, \tilde{\cW'})}[d] \\ {\rm
Nat}\big(\upphi^{\Sscript{t(g)}}, \tilde{\cW}\big) \ar@{->}^-{
\upphi^*(\tF)_{\Sscript{t(g)}} = {\rm
Nat}(\upphi^{\Sscript{t(g)}},\tilde{\tF})}[rr] & & {\rm
Nat}\big(\upphi^{\Sscript{t(g)}}, \tilde{\cW'}\big) }
\end{gather*}
of vector spaces, commutes. This follows from the equality
\begin{gather*}
{\rm Nat}\big(\upphi^{\Sscript{t(g)}},\tilde{\tF}\big) \circ {\rm Nat}\big(\upphi^{\Sscript{g}}, \tilde{\cW}\big) (\alpha)
 = \tilde{\tF} \circ \alpha \circ \upphi^{\Sscript{g}} = {\rm Nat}\big(\upphi^{\Sscript{g}}, \tilde{\cW'}\big) \circ {\rm Nat}\big(\upphi^{\Sscript{s(g)}},\tilde{\tF}\big) (\alpha),
\end{gather*}
for every $\alpha \in {\rm Nat}\big(\upphi^{\Sscript{s(g)}},\tilde{\cW}\big)$.
\end{proof}

\begin{Example}\label{exam:GM}Let $\cH:=(G\times X, X)$ be an action groupoid as in Example~\ref{exam:action} and consider the morphism $\upphi:={\rm pr}_{\Sscript{1}}\colon \cH \to G$ of groupoids, where~$G$ is considered as a groupoid with one object. Let $\cW \in \Rep{\cH}$ and set $\Gamma(\bara{\cW})$ the $\Bbbk$-vector space of global sections of the vector bundle $\bara{\cW}$ attached to the representation $\cW$. Then $\Gamma(\bara{\cW}) \in \Rep{G}$ and the assignment $\cW \to \Gamma(\bara{\cW})$ establishes a functor from $\Rep{\cH}$ to $\Rep{G}$. Furthermore, we have a natural isomorphism $\upphi^{\Sscript{*}}(\cW) \cong \Gamma(\bara{\cW})$, for every $\cH$-representation~$\cW$.
\end{Example}

\begin{Remark}\label{rem:HomSet} Let $\upphi\colon \cH \to \cG$ be a morphism of groupoids. Denote by ${}^{\Sscript{\upphi}}\uU(\cG):=\Ho \due \times {\Sscript{\upphi_0}}{ \Sscript{t}} \Ga$ the $(\cH,\cG)$-biset of Example \ref{exam:bisets}. Given an $\cH$-representation $\cW$ and consider it associated vector bundle $(\bara{\cW},\pi_{\Sscript{\cW}})$ as a left $\cH$-set by using the action of equation \eqref{Eq:actionV}. Then we have a natural isomorphism
\begin{gather*}
\xymatrix@R=0pt@C=20pt{ \hom{\cH\text{-}{\rm Sets}}{{
}^{\Sscript{\upphi}}\uU(\cG)}{\bara{\cW}} \ar@{->}^-{}[rr] & &
\underset{x \in \Go}{\prod} \upphi^{*}(\cW)_{\Sscript{x}},
\\ {\sf{f}} \ar@{|->}^-{}[rr] & & \big(
{\sf{f}}^{\Sscript{x}}\big)_{\Sscript{x \in \Go}}, \quad \text{where
} \big( {\sf{f}}^{\Sscript{x}}_{\Sscript{u}}\colon
\upphi^{\Sscript{x}}_{\Sscript{u}} \to \Wu, \big(a \mapsto
{\sf{f}}(u,a) \big) \big)_{\Sscript{u \in \Ho}}, \\
\big[ (u,a) \mapsto {\sf{p}}^{\Sscript{s(a)}}_{\Sscript{u}}(a)
\big]& & \big( {\sf{p}}^{\Sscript{x}}\big)_{\Sscript{x \in
\Go}}. \ar@{|->}^-{}[ll] }
\end{gather*}
This in fact comes from the natural isomorphism
\begin{gather*}
\xymatrix@R=0pt@C=20pt{ \hom{\cH\text{-}{\rm Sets}} {\vartheta(\{ x
\})^{-1}}{\bara{\cW}} \ar@{->}^-{}[r] & {\rm
Nat}\big(\upphi^{\Sscript{x}}, \tilde{\cW}\big) =
\upphi^{*}(\cW)_{\Sscript{x}},
\\ F \ar@{|->}^-{}[r] & \big[\upphi^{\Sscript{x}}_{\Sscript{u}} \to
W_{\Sscript{u}},\; \big( a \mapsto F(u, a)\big)\big], \\
\big[ (u,a) \mapsto \upeta_{\Sscript{u}}(a)\big] & \upeta,
\ar@{|->}^-{}[l] }
\end{gather*}
where $\vartheta\colon { }^{\Sscript{\upphi}}\uU(\cG) \to \Go$ is as before, the map $(u,a) \mapsto s(a)$.
\end{Remark}

\begin{Remark}[projection formula]\label{rem:PFromula}Analogue to the group case, see for instance~\cite[Proposition~2.10.18]{Kowalski:2014}, one can show that there is a natural isomorphism
\begin{gather*}
\upphi^{\Sscript{*}}\big( \cW \tensor{}\upphi_{\Sscript{*}}(\cV)\big) \cong \upphi^{\Sscript{*}}(\cW)\tensor{} \cV,
\end{gather*}
for any pair of representations $\cV \in \Rep{\cG}$ and $\cW \in \Rep{\cH}$. At the level of objects this isomorphism is given by
\begin{gather*}
\xymatrix@R=0pt{ {\rm Nat}\big(\upphi^{\Sscript{x}}, \tilde{\cW}\big) \tensor{\Bbbk} \Vx \ar@{->}^-{}[rr] & & {\rm Nat}\big(\upphi^{\Sscript{x}}, \widetilde{\cW
\tensor{}\upphi_{\Sscript{*}}(\cV)}\big), \\ \upeta\tensor{} {\bf{v}} \ar@{|->}^-{}[rr] & & \big[ \upphi^{\Sscript{x}}(u) \to \cW_{\Sscript{u}}\tensor{}\cV_{\Sscript{\phi(u)}}, \big( b \mapsto \upeta_{\Sscript{u}}(b)\tensor{}\cV^{\Sscript{b}}({\bf{v}})\big) \big]_{\Sscript{u \in \Ho}}, }
\end{gather*}
whose inverse is computed by fixing a dual basis $\{{\bf{v}}_j\}_{j} \subset \Vx$ and employing the dual basis $\{\cV^{\Sscript{b}}({\bf{v}}_j)\}_{j} \subset \cV_{\Sscript{\phi(u)}}$, for any $b \in \upphi^{\Sscript{x}}(u)$ and $u \in \Ho$. The rest of verifications are left to the reader.

A more conceptual proof can be given using Mitchell's Theorem \cite[Theorem 4.5.2]{Mitchell:1965} employing the set of small projective generators that enjoy both categories $\Rep{\cH}$ and $\Rep{\cG}$.
\end{Remark}

\subsection{The co-induction functor}\label{ssec:coInd}
Consider as before $\upphi\colon \cH \to \cG$ a morphism of groupoids. Denote by $\upg:=\Ga \due \times {\Sscript{s}}{ \Sscript{\upphi_0} } \Ho$ the underlying set of the $(\cG,\cH)$-biset described in Example~\ref{exam:bisets} with the two structure maps $ \varsigma\colon \upg \to \Go$, $(a,u) \mapsto t(a)$ and ${\rm pr}_{\Sscript{2}}\colon \upg \to \Ho$, $(a,u) \mapsto u$.

For a given $x \in\Go$, we consider the fibre $\varsigma^{-1}(\{x\})=\big\{ (a,u) \in \upg| t(a)=x \big\}$ as a~left $\cH$-invariant subset of $\upg$. We then consider its associated right translation groupoid $\varsigma^{-1}(\{x\}) \rJoin \cH$, which we denote by $\cH^{\Sscript{\upphi, x}}$, together with the canonical morphism of groupoids
\begin{gather*}
\vV^{\Sscript{\upphi, x}}\colon \ \cH^{\Sscript{\upphi, x}}= \varsigma^{-1}(\{x\}) \rJoin \cH \longrightarrow \cH, \qquad \big[ \big( ((a',u'),h), (a,u)\big) \longmapsto (h,u) \big],
\end{gather*}
that is, $\vV^{\Sscript{\upphi, x}}_{\Sscript{0}}(a,u)=u$ and $\vV^{\Sscript{\upphi, x}}_{\Sscript{1}}\big( ((a',u'),h) \big)=h$. As in Section~\ref{ssec:Res}, we have the attached restriction
functor $\vV^{\Sscript{\upphi, x}}_*\colon \Rep{\cH} \to \Rep{\cH^{\Sscript\upphi, x}}$, for any $x \in \Go$.
In this way, for any $x \in \Go$, and $\cW \in \Rep{\cH}$, we set\begin{gather}\label{Eq:SPx} {}^*\upphi(\cW)_{\Sscript{x}}: = \injlimit{\cH^{\Sscript{\upphi,
x}}}{\vV^{\Sscript{\upphi, x}}_*(\cW)}
\end{gather}
and denote by $\big\{\upupsilon^{\Sscript{x}}_{\Sscript{(a, u)}}\colon \vV^{\Sscript{\upphi, x}}_*(\cW)_{\Sscript{(a, u)}}=\cW_{\Sscript{u}} \longrightarrow {}^*\upphi(\cW)_{\Sscript{x}} \big\}_{\Sscript{(a, u) \in \varsigma^{-1}(\{x\})}}$ the structural $\Bbbk$-linear maps of this limit.

Now, given an arrow $g \in \Ga$, we have a diagram of morphism of groupoids
\begin{gather*}
\xymatrix{ \varsigma^{-1}(\{s(g)\}) \rJoin \cH \ar@{->}^-{}[d]
\ar@{->}@/^1pc/^-{\vV^{\Sscript{\upphi, s(g)}}}[drr] & & \\
\varsigma^{-1}(\{t(g)\}) \rJoin \cH \ar@{->}_-{\vV^{\Sscript{\upphi, t(g)}}}[rr] & & \cH, }
\end{gather*}
whose vertical arrow is the isomorphism sending $((a,u), h) \mapsto ( (ga,u),h)$. Therefore, by the definition of the limit of equation~\eqref{Eq:SPx}, there is a unique $\Bbbk$-linear map
${}^*\upphi(\cW)^{\Sscript{g}}\colon {}^*\upphi(\cW)_{\Sscript{s(g)}} \to {}^*\upphi(\cW)_{\Sscript{t(g)}}$ rendering commutative the following diagrams
\begin{gather}\label{Eq:lphiup}
\begin{gathered}
\xymatrix@R=25pt{ \pls(\cW)_{\Sscript{s(g)}}
\ar@{-->}^-{\pls(\cW)^g}[rr] & & \pls(\cW)_{\Sscript{t(g)}} \\
\cW_{\Sscript{u}}=\vV^{\Sscript{\upphi,
s(g)}}_*(\cW)_{\Sscript{(a, u)}}
\ar@{->}^-{\upupsilon^{\Sscript{s(g)}}_{\Sscript{(a, u)}}}[u]
\ar@{->}^-{\zeta^g_{(a, u)}}[rr] & &
\ar@{->}_-{\upupsilon^{\Sscript{t(g)}}_{\Sscript{(ga, u)}}}[u]
\vV^{\Sscript{\upphi, t(g)}}_*(\cW)_{\Sscript{(ga, u)}}
=\cW_{\Sscript{u}} }
\end{gathered}
\end{gather}
for every $(a,u) \in \varsigma^{-1}(\{s(g)\})$, where $\zeta^g_{(a, u)}$ acts by identity, that is,
\begin{gather*}
\zeta^g_{(a, u)}\colon \ \vV^{\Sscript{\upphi, s(g)}}_*(\cW)_{\Sscript{(a, u)}}=\cW_{\Sscript{u}}\longrightarrow \vV^{\Sscript{\upphi,t(g)}}_*(\cW)_{\Sscript{(ga, u)}}=\cW_{\Sscript{u}}, \qquad \big(w \longmapsto w \big).
\end{gather*}

Equations \eqref{Eq:SPx} and \eqref{Eq:lphiup} lead then to a functor
\begin{gather}\label{Eq:starphiup}
\begin{gathered}
\xymatrix@R=0pt{ \pls(\cW) \colon \ \cG \ar@{->}^-{}[rr] & & \Vect, \\ x \ar@{|->}^-{}[rr] & & \pls(\cW)_{\Sscript{x}}, \\ g \ar@{|->}[rr] & & \pls(\cW)^g. }
\end{gathered}
\end{gather}
On the other hand, if $\tF\colon \cW \to \cW'$ is a morphism of $\cH$-representations, then, for every $x \in \Go$, we define the following $\Bbbk$-linear map
\begin{gather}\label{Eq:pf}
{}^*\upphi(\tF)_{\Sscript{x}}\colon \ {}^*\upphi(\cW)_{\Sscript{x}} \longrightarrow {}^*\upphi(\cW')_{\Sscript{x}}
\end{gather}
as the unique $\Bbbk$-linear map which renders commutative the following diagram of $\Bbbk$-vector spaces
\begin{gather*}
\xymatrix@R=25pt{ \pls(\cW)_{\Sscript{x}} \ar@{-->}^-{\pls(\tF)_{\Sscript{x}}}[rr] & & \pls(\cW')_{\Sscript{x}} \\ \vV^{\Sscript{\upphi, x}}_*(\cW)_{\Sscript{(a, u)}}
\ar@{->}^-{\upupsilon^{\Sscript{\cW}}_{\Sscript{(a, u)}}}[u] \ar@{->}^-{\xi^{\Sscript{\tF}}_{\Sscript{(a, u)}}}[rr] & & \ar@{->}_-{\upupsilon^{\Sscript{\cW'}}_{\Sscript{(a, u)}}}[u]
\vV^{\Sscript{\upphi, x}}_*(\cW')_{\Sscript{(a, u)}}, }
\end{gather*}
where $\xi^{\Sscript{\tF}}_{\Sscript{(a, u)}}\colon \vV^{\Sscript{\upphi, x}}_*(\cW)_{\Sscript{(a, u)}} \to \vV^{\Sscript{\upphi, x}}_*(\cW')_{\Sscript{(a, u)}} $ is the $\Bbbk$-linear map sending $w \mapsto \tF_{\Sscript{u}}(w)$.

\begin{Lemma}\label{lema:Co-Induction}Let $\upphi\colon \cH \to \cG$ be a morphism of groupoids. Then the assignment ${}^{*}\upphi\colon \Rep{\cH} \to \Rep{\cG}$ described in equation~\eqref{Eq:starphiup}, gives rise to a well defined functor, referred to as {\rm the co-induction functor} of $\upphi$.
\end{Lemma}
\begin{proof}Given a morphism $\tF\colon \cW \to \cW'$ in the category $\Rep{\cH}$ and an arrow $g \in \Ga$, we need to check that the following diagram
\begin{gather*}
\xymatrix@C=45pt@R=35pt{ \pls(\cW)_{\Sscript{s(g)}} \ar@{->}^-{
\pls(\tF)_{\Sscript{s(g)}}}[rr]
\ar@{->}|-{\pls(\cW)^{\Sscript{g}}}[d] & &
\pls(\cW')_{\Sscript{s(g)}} \ar@{->}|-{\pls(\cW')^{\Sscript{g}}}[d]
\\ \pls(\cW)_{\Sscript{t(g)}} \ar@{->}^-{
\pls(\tF)_{\Sscript{t(g)}} }[rr] & & \pls(\cW')_{\Sscript{t(g)}} }
\end{gather*}
of vector spaces, commutes. This is equivalent to show the commutativity of the following diag\-rams:
\begin{gather*}
\xymatrix@C=45pt@R=25pt{ \vV^{\Sscript{\upphi,
s(g)}}_*(\cW)_{\Sscript{(a, u)}} \ar@{->}^-{
\xi^{\Sscript{\tF}}_{\Sscript{(a,u)}}}[rr]
\ar@{->}_-{\zeta_{\Sscript{(a,u)}}^{\Sscript{\cW, g}}}[d] & &
\vV^{\Sscript{\upphi, s(g)}}_*(\cW')_{\Sscript{(a, u)}}
\ar@{->}^-{\zeta_{\Sscript{(a,u)}}^{\Sscript{\cW', g}}}[d] \\
\vV^{\Sscript{\upphi, t(g)}}_*(\cW)_{\Sscript{(ga, u)}}
\ar@{->}^-{ \xi^{\Sscript{\tF}}_{\Sscript{(ga,u)}} }[rr] & &
\vV^{\Sscript{\upphi, t(g)}}_*(\cW')_{\Sscript{(ga, u)}} }
\end{gather*}
for any $(a, u) \in \varsigma^{-1}(\{s(g)\})$. However, this is immediate from the definitions of the involved maps.
\end{proof}

\begin{Remark}Given $\cW$ an $\cH$-representation, then one can consider, as in Section~\ref{ssec:biset}, the tensor product $\upg \tensor{\cH} \bara{\cW}$, where as above $(\bara{\cW},\pi_{\Sscript{\cW}})$ is the underlying vector bundle of $\cW$ endowed within its canonical left $\cH$-action, as in equation~\eqref{Eq:actionV}. An equivalence class of an element $((a,u),w) \in \upg \due \times {\Sscript{{\rm pr}_2}}{\; \Sscript{\pi}} \bara{\cW}$ will be denoted by $(a,u)\tensor{\cH}w$. Now, if we consider the vector bundle $\big(\bara{{}^*\upphi(\cW)}, \pi \big)$ endowed with its canonical left $\cG$-action, then we obtain the following $\cG$-equivariant map
\begin{gather*}
\xymatrix@R=0pt{ \upg \tensor{\cH} \bara{\cW} \ar@{->}^-{}[rr] & & \bara{{}^*\upphi(\cW)}, \\ (a,u)\tensor{\cH}w \ar@{|->}^-{}[rr] & &
\upupsilon^{\Sscript{t(a)}}_{\Sscript{(a,u)}} (w), }
\end{gather*}
which is not in general an isomorphism.
\end{Remark}

\section{Frobenius reciprocity formulae}\label{sec:FRF}
In this section we prove the left and the right Frobenius reciprocity formulae. This mainly establishes, from one side an adjunction between the restriction and induction functors, and from another one an adjunction between the restriction and co-induction functors. Form a~categorical point of view, this amounts to the notions of left and right Kan extensions. Here we follow an elementary and direct exposition, taking advantage of groupoids structure, without appealing to any heavy categorical notions.

\subsection{Right Frobenius reciprocity formula}\label{ssec:Adj}
In this subsection we show that the induction functor is a right adjoint functor to the restriction functor, that is, the right Frobenius reciprocity formula. So, let $\upphi\colon \cH \to \cG$ be a morphism of groupoids and consider two representations $\cV \in \Rep{\cG}$ and $\cW \in \Rep{\cH}$. Take a morphism $ \sigma \in \hom{\cH}{\upphi_{*}(\cV)}{\cW}$, and define, for every $x \in \Go$, the linear map
\begin{gather}
\Psi(\sigma)_{\Sscript{x}}\colon \ \Vx \longrightarrow \upphi^*(\cW)_{\Sscript{x}}={\rm Nat}\big(\phix, \tilde{\cW}\big),\nonumber\\
\hphantom{\Psi(\sigma)_{\Sscript{x}}\colon}{} \ \left( v \longmapsto \left[ \xymatrix@R=0pt{ \upphi^{\Sscript{x}}_{\Sscript{u}} =\cG(x,\phio(u)) \ar@{->}^-{}[r] & \cW_{\Sscript{u}} \\ p \ar@{|->}^-{}[r] & \sigma_{\Sscript{u}} (p v) } \right]_{u \in \Ho} \right),\label{Eq:Psigma}
\end{gather}
where $pv=\cV^{\Sscript{p}}(v) \in \cV_{\Sscript{\phio(u)}}$ is the left $\cG$-action on $\bara{\cV}$ defined in equation~\eqref{Eq:actionV}.

\begin{Lemma}\label{lema:Psi}The family of linear map $\{\Psi(\sigma)_{\Sscript{x}}\}_{x \in \Go}$ stated in equation~\eqref{Eq:Psigma} defines a natural transformation. That is, $\Psi(\sigma) \in \hom{\cG}{\cV}{\upphi^*(\cW)}$.
\end{Lemma}
\begin{proof}First let us check that each of the $\Psi(\sigma)_{\Sscript{x}}$'s is well defined. So given an arbitrary arrow $h \in \Ha$, we need to show that the diagram
\begin{gather}\label{Eq:digPhi}
\begin{gathered}
\xymatrix@C=40pt{ \phix_{\Sscript{s(h)}} \ar@{->}^-{\Psi(\sigma)_{\Sscript{x}}(v)_{\Sscript{s(h)}}}[rr] \ar@{->}_-{\phix_{\Sscript{h}}}[d] & & \ar@{->}^-{\cW^{\Sscript{h}}}[d] \cW_{\Sscript{s(h)}} \\ \phix_{\Sscript{t(h)}} \ar@{->}^-{\Psi(\sigma)_{\Sscript{x}}(v)_{\Sscript{t(h)}}}[rr] & & \cW_{\Sscript{t(h)}} }
\end{gathered}
\end{gather}
is commutative. So take an arrow $p \in \phix(s(h))= \cG(x,\phio(s(h))$, then
\begin{gather*}
h \Psi(\sigma)_{\Sscript{x}}(v)_{\Sscript{s(h)}}(p) = h \sigma_{\Sscript{s(h)}} (p v) =\sigma_{\Sscript{t(h)}} (\phia(h) p v),
\end{gather*}
because $\sigma$ is $\cH$-equivariant. On the other hand, we have that
\begin{gather*}
\Psi(\sigma)_{\Sscript{x}}(v)_{\Sscript{t(h)}} \circ \phix_{\Sscript{h}}(p) = \Psi(\sigma)_{\Sscript{x}}(v)_{\Sscript{t(h)}}(\upphi_{\Sscript{1}}(h)p) = \sigma_{\Sscript{t(h)}}(\upphi_{\Sscript{1}}(h)pv),
\end{gather*}
and this shows the commutativity of \eqref{Eq:digPhi}. Therefore, $\Psi(\sigma)_{\Sscript{x}}\colon \Vx \to \upphi^*(\cW)_{\Sscript{x}}$ is a well defined linear map. Now, take an arrow $ g \in \Ga$, we have to check that
\begin{gather*}
\xymatrix@C=40pt{ \cV_{\Sscript{s(g)}} \ar@{->}^-{\Psi(\sigma)_{\Sscript{s(g)}}}[rr] \ar@{->}_-{\cV^{\Sscript{g}}}[d] & & \ar@{->}^-{{\rm
Nat} (\upphi^{\Sscript{g}}, \tilde{\cW} )}[d] {\rm Nat}\big(\upphi^{\Sscript{s(g)}}, \tilde{\cW}\big) \\ \cV_{\Sscript{t(g)}} \ar@{->}^-{\Psi(\sigma)_{\Sscript{t(g)}}}[rr] & & {\rm Nat}\big(\upphi^{\Sscript{t(g)}}, \tilde{\cW}\big) }
\end{gather*}
is a commutative diagram of $\Bbbk$-vector spaces. To this end, consider an arbitrary $u \in \Ho$ and an arrow $p \in \upphi^{\Sscript{t(g)}}_{\Sscript{u}}=\cG(t(g), \upphi_{\Sscript{0}}(u))$. Then, for any vector $v \in \cV_{\Sscript{s(g)}}$, we have from one hand that
\begin{gather*}
\big( {\rm Nat}\big(\upphi^{\Sscript{g}}, \tilde{\cW}\big) \circ \Psi(\sigma)_{\Sscript{s(g)}}(v) \big)_{\Sscript{u}}(p) = \big(\Psi(\sigma)_{\Sscript{s(g)}}(v) \circ \upphi^{\Sscript{g}}\big)_{\Sscript{u}}(p) =\Psi(\sigma)_{\Sscript{s(g)}}(v)_{\Sscript{u}} \circ
\upphi^{\Sscript{g}}_{\Sscript{u}}(p) \\
\hphantom{\big( {\rm Nat}\big(\upphi^{\Sscript{g}}, \tilde{\cW}\big) \circ \Psi(\sigma)_{\Sscript{s(g)}}(v) \big)_{\Sscript{u}}(p)}{} = \Psi(\sigma)_{\Sscript{s(g)}}(v)_{\Sscript{u}} \big(pg\big) =
\sigma_{\Sscript{u}} \big(pg v\big),
\end{gather*}
and from the other one, we have that
\begin{gather*}
\big( \Psi(\sigma)_{\Sscript{t(g)}}(gv)\big)_{\Sscript{u}}(p) = \sigma_{\Sscript{u}}(p(gv)) = \sigma_{\Sscript{u}}(pgv),
\end{gather*}
whence the commutativity of that diagram.
\end{proof}

Reciprocally, take a morphism $\gamma \in \hom{\cG}{\cV}{\upphi^*(\cW)}$, then for every object $u \in\Ho$,
we set
\begin{gather*}
\Phi(\gamma)_{\Sscript{u}}\colon \ \cV_{\Sscript{\upphi_{0}(u)}} \longrightarrow \Wu, \qquad \big( v \longmapsto \gamma_{\Sscript{\upphi_{0}(u)}}(v)_{\Sscript{u}}(\iota_{\Sscript{\upphi_{0}(u)}})\big).
\end{gather*}

\begin{Lemma}\label{lema:Phi}The family of linear maps $\{ \Phi(\gamma)\}_{\Sscript{u \in \Ho}}$ defines a natural transformation. That is, we have a morphism $\Phi(\gamma) \in \hom{\cH}{\upphi_{*}(\cV)}{\cW}$.
\end{Lemma}
\begin{proof}Let $h \in \Ha$ and a vector $v \in \cV_{\Sscript{s(h)}}$. Then, from one hand, we have
\begin{gather}\label{Eq:vW}
\cW^{\Sscript{h}} \circ \Phi(\gamma)_{\Sscript{s(h)}}(v) = \cW^{\Sscript{h}} \circ \gamma_{\Sscript{\phi_0(s(h))}}(v)_{\Sscript{s(h)}}(\iota_{\Sscript{\upphi_{0}(s(h))}}) =\gamma_{\Sscript{\phi_0(s(h))}}(v)_{
\Sscript{t(h)}}(\upphi_{\Sscript{1}}(h)),
\end{gather}
where the second equality follows from the fact that $\gamma_{\Sscript{\phi_0(s(h))}}(v) \in {\rm Nat}\big(\upphi^{\Sscript{\phi_0(s(h))}}, \tilde{\cW}\big)$. On the other hand, since $\gamma \in \hom{\cG}{\cV}{\upphi^*(\cW)}$, we know that
\begin{gather}\label{Eq:up}
\gamma_{\Sscript{\upphi_0(t(h))}}\big( \upphi_{\Sscript{1}}(h) v \big)_{\Sscript{u}} (p) = \gamma_{\Sscript{\upphi_0(s(h))}}( v)_{\Sscript{u}} \circ \upphi^{\Sscript{\upphi_1(h)}}_{\Sscript{u}}(p),
\end{gather}
for every $u \in \Ho$ and $p \in \upphi^{\Sscript{\upphi_1(h)}} (u)=\cG( \upphi_{\Sscript{0}}(t(h)),\upphi_{\Sscript{0}}(u) )$, see equation~\eqref{Eq:Oslo}. Substituting in equation~\eqref{Eq:up}, $u= t(h)$ and $p = \iota_{\Sscript{\upphi_0(t(h))}}$, we then get
\begin{gather}\label{Eq:upp}
\gamma_{\Sscript{\upphi_0(t(h))}}\big( \upphi_{\Sscript{1}}(h) v
\big)_{\Sscript{t(h)}} (\iota_{\Sscript{\upphi_0(t(h))}}) =
\gamma_{\Sscript{\upphi_0(s(h))}}(v)_{\Sscript{t(h)}}
\circ \upphi^{\Sscript{\upphi_1(h)}}_{\Sscript{t(h)}}(\iota_{\Sscript{\upphi_0(t(h))}})
= \gamma_{\Sscript{\upphi_0(s(h))}}(v)_{\Sscript{t(h)}}
\big( \upphi_1(h) \big).
\end{gather}
Therefore,
\begin{gather*}
\Phi(\gamma)_{\Sscript{t(h)}}\circ \cV^{\Sscript{\phia(h)}} (v) = \Phi(\gamma)_{\Sscript{t(h)}}\big( \phia(h) v \big) = \gamma_{\Sscript{\upphi_0(t(h))}}\big( \phia(h) v\big)_{\Sscript{t(h)}}\big(\iota_{\Sscript{\upphi_0(t(h))}} \big)\\
\hphantom{\Phi(\gamma)_{\Sscript{t(h)}}\circ \cV^{\Sscript{\phia(h)}} (v)}{}\overset{\eqref{Eq:upp}}{=} \gamma_{\Sscript{\upphi_0(s(h))}}(v)_{\Sscript{t(h)}}
\big( \upphi_1(h) \big) \overset{\eqref{Eq:vW}}{=} \gamma_{\Sscript{\phi_0(s(h))}}(v)_{ \Sscript{t(h)}}(\upphi_{\Sscript{1}}(h)) = \cW^{\Sscript{h}} \circ \Phi(\gamma)_{\Sscript{s(h)}}(v),
\end{gather*}
for every $h \in \Ha$ and $v \in \cV_{\Sscript{s(h)}}$, and this finishes the proof.
\end{proof}

\begin{Proposition}[right Frobenius reciprocity]\label{prop:FrobeniusR} Let $\upphi\colon \cH \to \cG$ be a morphism of groupoids and consider the restriction $\upphi_{*}\colon \Rep{\cG} \to \Rep{\cH}$ and the induction $\upphi^{*}\colon \Rep{\cH} \to \Rep{\cG}$ functors. Then the maps $\Psi$ and $\Phi$ described, respectively, in Lemmas~{\rm \ref{lema:Psi}} and~{\rm \ref{lema:Phi}}, define a~natural isomorphism
\begin{gather*}
\xymatrix@R=0pt{ \hom{\cH}{\upphi_{*}(\cV)}{\cW} \ar@<1ex>@{->}^-{\Psi}[rr] & & \hom{\cG}{\cV}{\upphi^*(\cW)},\ar@<1ex>@{->}^-{\Phi}[ll] }
\end{gather*}
for every $\cG$-representation $\cV$ and $\cH$-representation $\cW$. In other words, the induction functor is a right adjoint functor of the restriction functor.
\end{Proposition}
\begin{proof}The naturality of both $\Psi$ and $\Phi$ are fulfilled by construction. Let us check that they are mutually inverse. So fixing $\sigma \in \hom{\cH}{\upphi_*(\cV)}{\cW}$ and $\gamma \in \hom{\cG}{\cV}{\upphi^*(\cW)}$, for every $u \in \Ho$ and $v \in \cV_{\Sscript{\phi_0(u)}}$, we have that
\begin{gather*}
\Phi\big( \Psi(\sigma)\big)_{\Sscript{u}}(v) =\Psi(\sigma)_{\Sscript{\phi_0(u)}}\big( v\big)_{\Sscript{u}}(\iota_{\Sscript{\upphi_0(u)}}) = \sigma_{\Sscript{u}}(v),
\end{gather*}
which implies that $\Phi \circ \Psi= {\rm id}$. On the other way around, for every $x \in \Go$, $w \in \Vx$, $u \in \Ho$ and $p \in \upphi^{\Sscript{x}}(u)=\cG(x,\upphi_{\Sscript{0}}(u))$, we have that
\begin{gather*}
\Psi\big( \Phi(\gamma)\big)_{\Sscript{x}}(w)_{\Sscript{u}}(p) = \Phi(\gamma)_{\Sscript{u}}(p w) =\gamma_{\Sscript{\phi_{\Sscript{0}}(u)}}\big( p w\big)_{\Sscript{u}}(\iota_{\Sscript{\phi_0(u)}}) =
 \big( {\rm Nat}\big(\upphi^{\Sscript{p}}, \tilde{\cW}\big) \circ \gamma_{\Sscript{x}}\big)(w)_{\Sscript{u}}(\iota_{\Sscript{\phi_0(u)}})\\
\hphantom{\Psi\big( \Phi(\gamma)\big)_{\Sscript{x}}(w)_{\Sscript{u}}(p)}{}
 = \gamma_{\Sscript{x}}(w)_{\Sscript{u}} \circ \upphi^{\Sscript{p}}_{\Sscript{u}}(\iota_{\Sscript{\phi_0(u)}}) = \gamma_{\Sscript{x}}(w)_{\Sscript{u}}\big(\iota_{\Sscript{\phi_0(u)}} p\big) =
\gamma_{\Sscript{x}}(w)_{\Sscript{u}}(p),
\end{gather*}
where in the third equality we have used the naturality of $\gamma$. Therefore, for every $x \in \Go$ and $w \in \Vx$, we have checked that $\Psi\big( \Phi(\gamma)\big)_{\Sscript{x}}(w)= \gamma_{\Sscript{x}}(w)$. This means that $\Psi \circ \Phi (\gamma)= \gamma$, for an arbitrary $\gamma$, which implies that $\Psi \circ \Phi = {\rm id}$ and this finishes the proof.
\end{proof}

\subsection{Left Frobenius reciprocity formula}\label{ssec:LAdj}

Keep the notations occurring in Section~\ref{ssec:coInd}. Next we proceed to show that the co-induction functor is a left adjoint functor of the restriction functor. To this end, the subsequent lemma is needed. We consider then a morphism of groupoids $\upphi\colon \cH \to \cG$.
\begin{Lemma}\label{lema:Gamma} Let $\cV$ and $\cW$ be, respectively, a $\cG$-representation and $\cH$-representation. For any morphism $\theta \in \hom{\cG}{\pls(\cW)}{\cV}$, the family of $\Bbbk$-linear maps:
\begin{gather*}
\bigg\{ \xymatrix@C=40pt{ \Gamma(\theta)_{\Sscript{u}}\colon \Wu =
\vV^{\Sscript{\upphi,\upphi(u)}}_*(\cW)_{\Sscript{(\iota_{\upphi(u)}, u)}}
\ar@{->}^-{\upupsilon^{\Sscript{\upphi(u)}}_{\Sscript{(\iota_{\upphi(u)},
 u)}} }[r] & \pls(\cW)_{\Sscript{\upphi(u)}}
\ar@{->}^-{\theta_{\Sscript{\upphi(u)}}}[r] & \cV_{\Sscript{\upphi(u)}} }\bigg\}_{u \in \Ho}
\end{gather*}
defines a morphism $\Gamma(\theta) \in \hom{\cH}{\cW}{\upphi_{*}(\cV)}$.
\end{Lemma}
\begin{proof}Given an arrow $h \in \Ha$, we set $\mathfrak{f}=\upupsilon^{\Sscript{\upphi(t(h))}}_{\Sscript{(\iota_{\upphi(t(h))},t(h))}} \circ\zeta^{\Sscript{\upphi(h)}}_{\Sscript{(\iota_{\phi(s(h))},s(h))}}$, see diagram~\eqref{Eq:lphiup}. Then, we have a commutative diagram
\begin{gather*}
\xymatrix@R=35pt{\vV^{\Sscript{\upphi,
\upphi(s(h))}}_*(\cW)_{\Sscript{(\iota_{\upphi(s(h))},s(h))}} =\cW_{\Sscript{s(h)}}
\ar@{->}^-{\upupsilon^{\Sscript{\upphi(s(h))}}_{\Sscript{(\iota_{\upphi(s(h))},s(h))}}}[rr] \ar@{->}_-{\Wh}[d] \ar@{->}^-{\mathfrak{f}}[drr] &
& \pls(\cW)_{\Sscript{\upphi(s(h))}}\ar@{->}^-{\theta_{\Sscript{\upphi(s(h))}}}[rr]\ar@{->}^-{\pls(\cW)^{\upphi(h)}}[d] & & \cV_
{\Sscript{\upphi(s(h))}} \ar@{->}^-{\cV^{\Sscript{\upphi(h)}}}[d]
\\ \vV^{\Sscript{\upphi, \upphi(t(h))}}_*(\cW)_{\Sscript{(\iota_{\upphi(t(h))}, t(h))}}=\cW_{\Sscript{t(h)}}
\ar@{->}_-{\upupsilon^{\Sscript{\upphi(t(h))}}_{\Sscript{(\iota_{\upphi(t(h))},t(h))}}}[rr] & & \pls(\cW)_{\Sscript{\upphi(t(h))}}\ar@{->}_-{\theta_{\Sscript{\upphi(t(h))}}}[rr] & & \cV_
{\Sscript{\upphi(t(h))}}, }
\end{gather*}
where the left hand square commutes, since the upper triangle is so by diagram \eqref{Eq:lphiup}, while the lower triangle commutes because of the limit defining $\pls(\cW)_{\Sscript{\upphi(t(h))}}$ and because the $\zeta$ map acts by identities. This shows that $\Gamma(\theta)\colon \cW \to \upphi_{*}(\cV)$ is a natural transformation, as desired.
\end{proof}

In the other way around, consider $\delta \in \hom{\cH}{\cW}{\upphi_{*}(\cV)}$. For a fixed object $x \in \Go$, we set the following family of $\Bbbk$-linear maps
\begin{gather}\label{Eq:S}
\Big\{ \xymatrix{ \upsigma^{\Sscript{x}}_{\Sscript{(a, u)}}\colon \vV^{\Sscript{\upphi, x}}_*(\cW)_{\Sscript{(a, u)}} = \Wu
\ar@{->}^-{\delta_u}[r] & \cV_{\Sscript{\upphi(u)}}
\ar@{->}^-{\cV^{\Sscript{a}}}[r] & \Vx } \Big\}_{\Sscript{(a,u) \in \cH^{\Sscript{\upphi, x}}}},
\end{gather}
where $\cH^{\Sscript{\upphi, x}}$ is, as in Section~\ref{ssec:coInd}, the right translation groupoid $\varsigma^{-1}(\{x\})\rJoin \cH$. It is from it own definition that the family $\big\{ \upsigma^{\Sscript{x}}_{\Sscript{(a,
u)}}\big\}_{\Sscript{(a, u) \in \cH^{\Sscript{\upphi, x}}}}$ is an inductive system of $\Bbbk$-vector spaces. Therefore, for every $x \in \Go$, there is a unique $\Bbbk$-linear map
\begin{gather}\label{Eq:SI}
\Sigma(\delta)_{\Sscript{x}} = \injlimit{(a, u) \in \cH^{\Sscript{\upphi,x}}}{\upsigma^{\Sscript{x}}_{\Sscript{(a, u)}}} \colon \
{}^*\upphi(\cW)_{\Sscript{x}} \longrightarrow \Vx
\end{gather}
such that $\Sigma(\delta)_{\Sscript{x}} \circ \upupsilon^{\Sscript{x}}_{\Sscript{(a, u)}} = \upsigma^{\Sscript{x}}_{\Sscript{(a, u)}}$, for any object $(a,u)$ in $\cH^{\Sscript{\upphi,x}}$. Furthermore, given an arrow $g \in
\Ga$ and an object $(a,u)$ in $\cH^{\Sscript{\upphi,\ x}}$, we have a~commutative diagram
\begin{gather}\label{Eq:SII}
\xymatrix@R=25pt{ \cV_{\Sscript{s(g)}} \ar@{->}^-{\Vg}[rr] & &
\cV_{\Sscript{t(g)}} \\ \vV^{\Sscript{\upphi,
s(g)}}_*(\cW)_{\Sscript{(a, u)}}
\ar@{->}^-{\upsigma^{\Sscript{s(g)}}_{\Sscript{(a, u)}}}[u]
\ar@{->}^-{\zeta^g_{(a, u)}}[rr] & &
\ar@{->}_-{\upsigma^{\Sscript{t(g)}}_{\Sscript{(ga, u)}}}[u]
\vV^{\Sscript{\upphi, t(g)}}_*(\cW)_{\Sscript{(ga, u)}}, }
\end{gather}
where $\zeta^g_{(a, u)}$ is as in diagram \eqref{Eq:lphiup}.

\begin{Lemma}\label{lema:Sigma}Let $\cW$ and $\cV$ be, respectively, an $\cH$-representation and a~$\cG$-representation with a morphism $\delta \in \hom{\cH}{\cW}{\upphi_{*}(\cV)}$. Then the family of $\Bbbk$-linear
maps $\big\{ \Sigma(\delta)_{\Sscript{x}} \big\}_{\Sscript{ x \in \Go}}$ of equation~\eqref{Eq:SI}, defines a morphism $\Sigma(\delta) \in \hom{\cG}{{}^*\upphi(\cW)}{\cV}$.
\end{Lemma}
\begin{proof}For any arrow $g \in \Ga$ and an object $(a,u)$ in $\cH^{\Sscript{\upphi, s(g)}}$, we have that the diagram
\begin{gather*}
\xymatrix@R=30pt{ & \cV_{\Sscript{s(g)}} \ar@{->}^-{\Vg}[rrr] & & &
\cV_{\Sscript{t(g)}} \\ {}^*\upphi(\cW)_{\Sscript{s(g)}}
\ar@{->}^-{{}^*\upphi(\cW)^g}[rrr]
\ar@{->}^-{\Sigma(\delta)_{\Sscript{s(g)}}}[ur] & & &
{}^*\upphi(\cW)_{\Sscript{t(g)}}
\ar@{->}^-{\Sigma(\delta)_{\Sscript{t(g)}}}[ur] & \\
\vV^{\Sscript{\upphi, s(g)}}_*(\cW)_{\Sscript{(a, u)}}
\ar@{->}@/_1pc /_->>>>{\upsigma^{\Sscript{s(g)}}_{\Sscript{(a,
u)}}}[uur] \ar@{->}^-{\zeta^g_{(a, u)}}[rrr]
\ar@{->}^-{\upupsilon^{\Sscript{s(g)}}_{\Sscript{(a, u)}}}[u] & &
& \ar@{->}^-{\upupsilon^{\Sscript{t(g)}}_{\Sscript{(ga, u)}}}[u]
\ar@{->}@/_1pc /_->>>>{\upsigma^{\Sscript{t(g)}}_{\Sscript{(ga,
u)}}}[uur] \vV^{\Sscript{\upphi, t(g)}}_*(\cW)_{\Sscript{(ga,
u)}} & }
\end{gather*}
commutes by the definition of the involved maps and diagram \eqref{Eq:SII}. Therefore, the upper square should commutes as~well, and this shows that $\Sigma(\delta)\colon {}^*\upphi(\cW) \to \cV $ is a morphism of $\cG$-representations, as claimed.
\end{proof}

\begin{Proposition}[left Frobenius reciprocity]\label{prop:FrobeniusL} Let $\upphi\colon \cH \to \cG$ be a morphism of groupoids and consider the restriction $\upphi_{*}\colon \Rep{\cG} \to \Rep{\cH}$ and the co-induction ${}^{*}\upphi\colon \Rep{\cH} \to \Rep{\cG}$ functors. Then the maps $\Gamma$ and $\Sigma$ described, respectively, in Lemmas~{\rm \ref{lema:Gamma}} and~{\rm \ref{lema:Sigma}}, define a natural isomorphism
\begin{gather*}
\xymatrix@R=0pt{ \hom{\cH}{\cW}{\upphi_{*}(\cV)}
\ar@<1ex>@{->}^-{\Sigma}[rr] & & \hom{\cG}{{}^*\upphi(\cW)}{\cV},
\ar@<1ex>@{->}^-{\Gamma}[ll] }
\end{gather*}
for every $\cG$-representation $\cV$ and $\cH$-representation $\cW$. In other words, the co-induction functor is a left adjoint functor of the restriction functor.
\end{Proposition}
\begin{proof}For a given $\delta \in\hom{\cH}{\cW}{\upphi_{*}(\cV)} $ and $u \in\Ho$, we know from Lemma~\ref{lema:Gamma} that
\begin{gather*}
\xymatrix@C=40pt{ \Gamma(\Sigma(\delta)){\Sscript{u}}\colon \vV^{\Sscript{\upphi,
\upphi(u)}}_*(\cW)_{\Sscript{(\iota_{\upphi(u), u})}}
\ar@{->}^-{\upupsilon^{\Sscript{\upphi(u)}}_{\Sscript{(\iota_{\upphi(u)},
u)}}}[r] & {}^*\upphi(\cW)_{\Sscript{\upphi(u)}}
\ar@{->}^-{\Sigma(\delta)_{\Sscript{\phi(u)}}}[r] &
\cV_{\Sscript{\upphi(u)}}. }
\end{gather*}
Therefore,
\begin{gather*}
\Gamma(\Sigma(\delta)){\Sscript{u}} = \Sigma(\delta)_{\Sscript{\phi(u)}} \circ\upupsilon^{\Sscript{\upphi(u)}}_{\Sscript{(\iota_{\upphi(u)},
u)}} \overset{\eqref{Eq:SI}}{=}
\upsigma^{\Sscript{\upphi(u)}}_{\Sscript{(\iota_{\upphi(u)}, u)}}
 = \delta_{\Sscript{u}}
\end{gather*}
for every object $u \in\Ho$, and so $\Gamma \circ \Sigma = {\rm id}$. Conversely, starting with a morphism $\theta \in \hom{\cG}{\pls(\cW)}{\cV}$ and an object $x \in \Go$, we have that
\begin{gather*}
\Sigma( \Gamma(\theta))_{\Sscript{x}} \overset{\eqref{Eq:S},\, \eqref{Eq:SI}}{=} \injlimit{(a, u) \in \cH^{\upphi,
x}}{\cV^{\Sscript{a}} \circ \Gamma(\theta)_{\Sscript{u}}}
 \overset{\ref{lema:Gamma}}{=} \injlimit{(a, u) \in\cH^{\upphi, x}}{ \cV^{\Sscript{a}} \circ \theta_{\Sscript{\upphi(u)}} \circ \upupsilon^{\Sscript{\upphi(u)}}_{(\iota_{\upphi(u)}, u)} }\\
\hphantom{\Sigma( \Gamma(\theta))_{\Sscript{x}}}{}
 = \injlimit{(a, u) \in \cH^{\upphi, x}}{\theta_{\Sscript{x}} \circ {}^*\upphi(\cW)^{\Sscript{g}} \circ
\upupsilon^{\Sscript{\upphi(u)}}_{(\iota_{\upphi(u)}, u)} }
= \injlimit{(a, u) \in \cH^{\upphi, x}}{
\theta_{\Sscript{x}} \circ \upupsilon^{\Sscript{x}}_{(a, u)}} \\
\hphantom{\Sigma( \Gamma(\theta))_{\Sscript{x}}}{}
= \theta_{\Sscript{x}} \circ \injlimit{(a, u) \in \cH^{\upphi, x}}{ \upupsilon^{\Sscript{x}}_{(a, u)} }
 = \theta_{\Sscript{x}},
\end{gather*}
where in the third equality we have used the naturality of $\theta$ and in the fourth one the diagram~\eqref{Eq:lphiup}. This shows that $\Sigma \circ \Gamma = {\rm id}$ and finishes the proof.
\end{proof}

\begin{Remark}\label{rem:LRKan} As the expertise reader can observe, the construction of the induction and co-induction functors $\upphi^{*}$ and ${}^{*}\upphi$ corresponds, respectively, (up to natural isomorphisms) to the well known universal construction of the right and left Kan extensions of the functor $\upphi_{*}$, see~\cite{MacLane} for more details. The proof presented here is somehow elementary and makes use of the groupoid structure, for instance the notion of translation groupoid among others. This also have the advantage of describing explicitly the natural isomorphisms establishing these adjuntions, which in fact is crucial to follow the arguments of the main result of the paper stated in the forthcoming section.
\end{Remark}

\section{Frobenius extensions in groupoids context}\label{sec:Frobenius}
The main aim of this section is to characterize Frobenius morphism of groupoids, see Definition~\ref{def:Frobenius} below. This in fact is a kind of an universal definition which can be applied to any functor with left and right adjoints functors. Typical examples are Frobenius algebras over a given field (or commutative ring), where the forgetful functor form the category of modules to vector spaces has isomorphic left and right adjoint functors, namely, the tensor and the homs functors (see~\cite{Kadison:1999}). Our main result can be seen also as an approach to Frobenius extensions of algebras with enough orthogonal idempotents.

\begin{Definition}\label{def:Frobenius}Let $\upphi \colon \cH \to \cG$ be a morphism of groupoids. We say that $\upphi$ is a \emph{Frobenius morphism} provided that the induction and the co-induction functors $\upphi^{*}$ and ${}^{*}\upphi$ are naturally isomorphic.
\end{Definition}

From now on, we will freely use the notations and the notions expounded in Section~\ref{ssec:notation}. So let us consider a morphism $\upphi\colon \cH \to \cG$ of groupoids, and denote by $\phi\colon A \to B$ the associated morphism of $\Bbbk$-vector space defined in equation~\eqref{Eq:phi}. If we assume that $\phio\colon \Go \to \Go$ is injective, then $\phi$ becomes a morphism of rings with enough orthogonal idempotents. The following counterexample shows that $\phi$ could not be multiplicative, without assuming $\phio$ injective:

\begin{Example}\label{exam:Phi} Let $f\colon X \to Y$ be a non injective map, by choosing two distinct elements $x, x' \in X$ whose images are equal $f(x)=f(x')$. Take $\cH:=(X,X)$ to be a trivial groupoid as in Example~\ref{exam:trivial}, and $\cG=(Y\times Y, Y)$ to be a groupoid of pairs as in Example~\ref{exam:X}(1). It is by definition that, for any two objects $y$, $y'$ in~$\cG$, there is only one arrow from $y$ to $y'$, namely, the one defined by the pair $(y,y')$. We denote by $\bd{1}_{\Sscript{(y, y')}}$ the image of this arrow in the ring $B$, so that we have $\bd{1}_{\Sscript{(z, z)}} = 1_{\Sscript{z}}$, for any $z \in Y$. In this situation, the associated rings with enough orthogonal idempotents are the direct sums of the form $A=\Bbbk^{(X)}$ and $B=\Bbbk^{(Y \times Y)}$, respectively. Consider the functor $\upphi\colon \cH \to \cG$ whose arrows map is $\phia\colon X \to Y \times Y$, which sends $\iota_{\Sscript{x}} \mapsto \iota_{\Sscript{(f(x), f(x))}}$, and its objects map is given by~$\phio=f$. Therefore, the $\Bbbk$-linear map $\phi\colon \Bbbk^{(X)} \to \Bbbk^{(Y \times Y)}$ attached to~$\upphi$ sends $1_{\Sscript{u}} \mapsto \bd{1}_{\Sscript{(f(u), f(u))}}$, for any $u \in X$. Now coming back to the chosen elements $x\neq x'$, we know that $\phi(1_{\Sscript{x'}}) = 1_{\Sscript{\phio(x')}}=1_{\Sscript{f(x')}}= 1_{\Sscript{f(x)}}=1_{\Sscript{\phio(x)}} =\phi(1_{\Sscript{x}})$, so that we get
\begin{gather*}
\phi( 1_{\Sscript{x}} . 1_{\Sscript{x'}}) = \phi(0) = 0 \qquad \text{and} \qquad \phi( 1_{\Sscript{x}}) . \phi(1_{\Sscript{x'}}) =1_{\Sscript{f(x)}} . 1_{\Sscript{f(x')}} =1_{\Sscript{f(x)}} \neq 0,
\end{gather*}
which show that $\phi( 1_{\Sscript{x}} . 1_{\Sscript{x'}}) \neq \phi( 1_{\Sscript{x}}) . \phi(1_{\Sscript{x'}})$, and so $\phi$ is not multiplicative.
\end{Example}

By scalar restriction, $B$ is considered as an $A$-bimodule, although, this is not necessarily an unital one. The underlying vector spaces of the unital left, right $A$-module and $A$-bimodule parts of $B$ are, respectively, given by the direct sums:
\begin{gather*}
AB = \bigoplus_{u \in \Ho, x \in \Go} \Bbbk \cG(x, \phio(u)),\qquad BA = \bigoplus_{u \in \Ho, x \in \Go} \Bbbk \cG(\phio(u), x), \nonumber\\
ABA = \bigoplus_{u, u' \in \Ho} \Bbbk \cG(\phio(u),\phio(u')).
\end{gather*}
Since we are in groupoids context, it is clear that $AB \cong BA$ as $\Bbbk$-vector spaces. We refer to~\cite{ElKaoutit:2009}, for more details on unital modules and on the notion of finitely generated and projective unital modules over rings with local units, specially their characterization by means of tensor and homs functors.

The following is our main result:
\begin{Theorem}\label{thm:FM}Let $\upphi\colon \cH \to \cG$ be a morphism of groupoids and consider as above the associated algebras $A$ and $B$, respectively. Assume that $\upphi_{\Sscript{0}}$ is an injective map. Then the following are equivalent.
\begin{enumerate}[$(i)$]\itemsep=0pt
\item $\upphi$ is a Frobenius morphism;
\item There exists a natural transformation $\Sf{E}_{\Sscript{(u, v)}}\colon \cG(\phio(u),\phio(v)) \longrightarrow \Bbbk \cH(u,v)$ in $\cH^{\rm op} \times \cH$, and for every $x \in \Go$, there exists a~finite set
$\{\big((u_{i},b_{i}), c_{i}\big)\}_{i=1, \dots, N} \in \varsigma^{-1}\big( \{x\}\big)\times \Bbbk \cG(x, \phio(u_{i}))$ such that, for every pair of elements $ (b, b') \in \cG(x, \phio(u)) \times \cG(\phio(u), x)$, we have
\begin{gather*}
\sum_{i}\Sf{E}(bb_{i})c_{i} = b \in \Bbbk \cG(x, \phio(u)) \qquad \text{and} \qquad b' = \sum_{i}b_{i}\Sf{E}(c_{i}b') \in \Bbbk \cG(\phio(u), x).
\end{gather*}
\item For every $x\in \Go$, the left unital $A$-module $AB1_{x}$ is finitely generated and projective and there is a natural isomorphism $B1_{u} \cong B\hom{A-}{AB}{A1_{u}}$, of left unital $B$-modules, for every $u\in \Ho$.
\end{enumerate}
\end{Theorem}

The proof of this theorem will be done in several steps, following the path: $(i) \Rightarrow (ii) \Rightarrow (iii) \Rightarrow (i)$.

\subsection[The proof of $(i) \Rightarrow (ii)$ in Theorem~\ref{thm:FM}]{The proof of $\boldsymbol{(i) \Rightarrow (ii)}$ in Theorem~\ref{thm:FM}}\label{sssec:I-II}
Assume that there is a natural isomorphism ${}^{*}\upphi \cong \upphi^{*}$. This in particular implies that, for every $x \in \Go$, there is a natural isomorphism
\begin{gather}\label{Eq:I-II}
{\rm Nat}(\upphi^{\Sscript{x}}, \tilde{(-)}) \cong \injlimit{\cH^{\Sscript{\upphi, x}}}{\vV^{\Sscript{\upphi, x}}_*(-)}.
\end{gather}
Now, for a given object $u \in \Ho$, let us denote by $\cH_{u}\colon \cH \to \Vect$ the $\cH$-representation given over objects by $v \to \Bbbk \cH(u, v)$ and obviously defined over arrows. The natural isomorphism of~\eqref{Eq:I-II}, leads then to a family of bijections
\begin{gather*}
{\rm Nat}\big(\upphi^{\Sscript{x}}, \tilde{\cH_{u}}\big) \cong \injlimit{\cH^{\Sscript{\upphi, x}}}{\vV^{\Sscript{\upphi, x}}_*(\cH_{u})},
\end{gather*}
which is clearly natural in $u \in \Ho$. In the previous situation, we have
\begin{Lemma}\label{lema:I-II-1} For any $x \in\Go$ and $u \in \Ho$, there is an isomorphism
\begin{gather*}
\injlimit{\cH^{\Sscript{\upphi, x}}}{\vV^{\Sscript{\upphi,x}}_*(\cH_{u})} \cong \Bbbk \cG(\phio(u),x),
\end{gather*}
which is natural in both components $(u, x) \in \cH^{\Sscript{\rm op}} \times \cG$.
\end{Lemma}
\begin{proof}Fix for the moment $x$ and $u$ as in the statement. We need to check that the vector space $\Bbbk \cG(\phio(u), x)$ is the inductive limit of the inductive system
\begin{gather*}
\big\{ \vV^{\Sscript{\upphi, x}}_*(\cH_{u})_{\Sscript{(b, v)}}=\Bbbk \cH(v,u) \big\}_{\Sscript{(b, v) \in \varsigma^{-1}(\{x\})}}.
\end{gather*}
Let us first define an inductive cone over $\Bbbk \cG(\phio(u), x)$. So take $(b,v) \in \varsigma^{-1}(\{x\})$, that is, $b \in \cG(\phio(v), x)$, then we have a linear map
\begin{gather*}
\tau_{\Sscript{(b, v)}}\colon \ \Bbbk \cH(u,v) \longrightarrow \Bbbk\cG(\phio(u), x),\qquad \bigg( \sum_{i}\lambda_{i} a_{i} \longmapsto \sum_{i}\lambda_{i} b\phia(a_{i})\bigg),
\end{gather*}
which is clearly compatible with the arrows of the groupoid $\cH^{\Sscript{\upphi, x}}$, and this gives us the desired inductive cone. We need then to check that this is an initial object among all others cones. So given an arbitrary cone $\big\{\xi^{\Sscript{x}}_{\Sscript{(b, v)}}\colon \Bbbk\cH(u,v) \to V\big\}_{\Sscript{(b, v)} \in \varsigma^{-1}(\{x\}) }$, for any $b \in \cG(\phio(u), x)$, we can consider the vector $\xi^{\Sscript{x}}_{\Sscript{(b, u)}}(1_{u}) \in V$; whence we have a linear map
\begin{gather*}
\xi_{ }\colon \ \Bbbk \cG(\phio(u),x) \longrightarrow V, \qquad \big( b \longmapsto \xi^{\Sscript{x}}_{\Sscript{(b, u)}}(1_{u}) \big).
\end{gather*}
It turns out that this is a morphism of inductive cones and this finishes the proof of the lem\-ma.
\end{proof}

By Lemma \ref{lema:I-II-1}, the natural isomorphisms of equation~\eqref{Eq:I-II}, lead to a natural isomorphism
\begin{gather*}
{\rm Nat}\big(\upphi^{\Sscript{x}}, \tilde{\cH_{u}}\big) \cong \Bbbk \cG(\phio(u),x),
\end{gather*}
for every $x \in \Go$ and $u \in \Ho$. The following general lemma characterizes these kind of natural transformations.

\begin{Lemma}\label{lema:I-II-2}Let $F\colon \mathcal{C} \to \cD$ be a covariant functor between small categories. Then there is a~natural isomorphism
\begin{gather*}
{\rm Nat}\big[ \cD(F(-),+), {\rm Nat} \big( \cD(+, F(\star)), \Bbbk \mathcal{C}(-,\star) \big) \big] \cong {\rm Nat}\big( \cD(F(\dag),F(\ddag)), \Bbbk \mathcal{C}(\dag,\ddag) \big),
\end{gather*}
where the left hand side term stands for the set of all natural transformations of the form:
\begin{gather*}
\zeta_{(c, d)}\colon \ \cD(F(c),d) \longrightarrow{\rm Nat} \big( \cD(d, F(\star)), \Bbbk \mathcal{C}(c,\star) \big), \qquad (c,d) \in \mathcal{C}^{\Sscript{\rm op}} \times \cD,
\end{gather*}
and the right hand side term is the set of all natural transformations of the form:
\begin{gather*}
\cE_{(c', c'')} \colon \ \cD(F(c'), F(c'')) \longrightarrow \Bbbk \mathcal{C}(c',c''), \qquad (c',c'') \in \mathcal{C}^{\Sscript{\rm op}} \times \mathcal{C}.
\end{gather*}
\end{Lemma}
\begin{proof}
The stated natural isomorphism is given by the following isomorphism:
\begin{gather*}
\xymatrix@R=0pt{ {\rm Nat}\big[ \cD(F(-),+), {\rm Nat} \big( \cD(+, F(\star)), \Bbbk \mathcal{C}(-,\star) \big) \big] \ar@<0.90ex>@{->}^-{\Omega}[rr] & & \ar@<0.90ex>@{->}^-{\Gamma}[ll] {\rm Nat}\big( \cD(F(\dag),F(\ddag)), \Bbbk \mathcal{C}(\dag,\ddag) \big),}
\end{gather*}
where, for a given natural transformation $\zeta$ in the domain of $\Omega$, we have
\begin{gather*}
\Omega(\zeta)_{(c',c'')} \colon \ \cD(F(c'),F(c'')) \longrightarrow \Bbbk \mathcal{C}(c',c''), \qquad \big( p \longmapsto \zeta_{(c', F(c''))}(p)_{c''}(1_{F(c'')}) \big).
\end{gather*}
As for a natural transformation $\cE$ in the domain of $\Gamma$, we have that
\begin{gather*}
\Gamma(\cE)_{(c, d)}\colon \ \cD(F(c),d) \longrightarrow {\rm Nat} \big( \cD(d, F(\star)), \Bbbk \mathcal{C}(c,\star) \big), \qquad \big( q \longmapsto \Gamma(\cE)_{(c, d)}(q) \big),
\end{gather*}
which assigns to every object $c' \in \mathcal{C}$, the map
\begin{gather*}
\Gamma(\cE)_{(c, d)}(q) { }_{c'}\colon \ \cD(d,F(c')) \longrightarrow \Bbbk \mathcal{C}(c,c'), \qquad \big[ r \mapsto \cE_{(c, c')}(rq) \big].
\end{gather*}
The rest of the proof is left to the reader.
\end{proof}

Applying Lemma \ref{lema:I-II-2} to our case, we know that any natural transformation
\begin{gather*}
\Theta_{\Sscript{(u,x)}}\colon \ \cG(\phio(u),x) \longrightarrow {\rm Nat}\big( \upphi^{\Sscript{x}}, \tilde{\cH_{u}}\big)
\end{gather*}
gives rise to a natural transformation
\begin{gather}\label{Eq:E}
{\sf{E}}_{\Sscript{(v,w)}}\colon \ \cG\big(\phio(v),\phio(w)\big) \longrightarrow \Bbbk\cH(v,w), \qquad \big( b \longmapsto \Theta_{\Sscript{(v,\phio(w))}}(b){}_{w}(1_{\phio(w)}) \big).
\end{gather}

On the other hand, the left hand functor in the isomorphism of equation~\eqref{Eq:I-II}, that is, the functor ${\rm Nat}\big(\upphi^{\Sscript{x}}, \tilde{(-)}\big)$ should preserves colimits, since the involves categories are Grothendieck ones with a set of small projective generators. Namely, in the case of $\Rep{\cH}$ this set of generators, is given by the family of representations $\{ \cH_{\Sscript{u}}\}_{\Sscript{u \in \Ho}}$. In this way, saying that ${\rm Nat}\big(\upphi^{\Sscript{x}}, \tilde{(-)}\big)$ preserves colimits, is equivalent to say that the $\cH$-representation $\Bbbk \upphi^{\Sscript{x}}\colon \cH \to \Vect$, $u \to \Bbbk \upphi^{\Sscript{x}}_{\Sscript{u}}=\Bbbk\cG(x,\phio(u))$, is finitely generated and projective. Therefore, we are assuming that there exists a positive integer $N \geq 1$ and a split monomorphism $\Bbbk \upphi^{\Sscript{x}} \hookrightarrow \oplus_{i =1,\dots, N} \cH_{\Sscript{u_{i}}}$. Hence, there is a set $\{ c_{i}\}_{\Sscript{i =1,\dots, N}}$ such that each of the $c_{i}$'s belong to $\Bbbk \cG(x, \phio(u_{i}))$, and natural transformations $\varphi{}^{i}_{-}\colon \Bbbk \cG(x,\phio(-)) \to \Bbbk \cH(u_{i},-)$, such that, for every $b \in \cG(x,\phio(u))$, we have that
\begin{gather*}
\sum_{i} \phi\big( \varphi^{i}_{u}(b) \big) c_{i} = b,
\end{gather*}
equality in the vector space $ \Bbbk\cG(x,\phio(u))$. In this direction, one can consider the family of elements $\{b_{i}\}_{\Sscript{i= 1,\dots, N}}$ where each $b_{i} = \Theta_{(u_{i}, x)}^{-1}\big( \varphi^{i}_{-}\big) \in \cG(\phio(u_{i}),x)$. Using the natural transformations~$\Sf{E}$ of~\eqref{Eq:E} together with the properties which $\Theta$ satisfies, we obtain
\begin{gather*}
\sum_{i} \Sf{E}_{\Sscript{(u_{i}, u)}} \big( bb_{i}\big) c_{i} = b,
\end{gather*}
for every $b$ as above. Take an element $b' \in \cG(\phio(u), x)$, then for every $c \in \cG(x,\phio(v))$ with $v \in \Ho$, we have that
\begin{align*}
\Theta_{\Sscript{(u, x)}} \bigg( \sum_{i} b_{i} \Sf{E}_{\Sscript{(u, u_{i})}} \big( c_{i} b'\big)\bigg)_{\Sscript{v}}(c) &= \sum_{i} \Theta_{\Sscript{(u, x
)}}(b_{i})_{\Sscript{v}}(c) \Sf{E}_{\Sscript{(u, u_{i})}} \big(c_{i} b'\big) \\
 &= \sum_{i} \Theta_{\Sscript{(u, x )}}\big(\Theta_{(u_{i}, x)}^{-1}\big( \varphi^{i}_{-}\big)\big)_{\Sscript{v}}(c) \Sf{E}_{\Sscript{(u, u_{i})}} \big( c_{i} b'\big) \\
&= \sum_{i} \varphi^{i}_{\Sscript{v}}(c) \Sf{E}_{\Sscript{(u,u_{i})}} \big( c_{i} b'\big) = \sum_{i}\Sf{E}_{\Sscript{(u, v)}}\big( \varphi^{i}_{\Sscript{v}}(c) c_{i}b'\big)\\
&= \Sf{E}_{\Sscript{(u, v)}} \bigg( \sum_{i}\varphi^{i}_{\Sscript{v}}(c) c_{i} b'\bigg) =\Sf{E}_{\Sscript{(u, v)}} \big( cb'\big) \\
&=\Theta_{\Sscript{(u,\upphi(v))}}(cb'){}_{v}(1_{\upphi(v)})= \Theta_{\Sscript{(u,x)}}(b'){}_{v}(c).
\end{align*}
Thus, $\Theta_{\Sscript{(u, x)}} \Big( \sum_{i} b_{i} \Sf{E}_{\Sscript{(u, u_{i})}} \big( c_{i} b'\big) \Big) = \Theta_{\Sscript{(u,x)}}(b')$ and so $\sum_{i} b_{i} \Sf{E}_{\Sscript{(u, u_{i})}} \big( c_{i} b'\big) = b'$, which completes the proof of $(ii)$. \QEDA

\subsection[The proof of $(ii) \Rightarrow (iii)$ in Theorem \ref{thm:FM}]{The proof of $\boldsymbol{(ii) \Rightarrow (iii)}$ in Theorem \ref{thm:FM}}\label{sssec:II-III}
For every $x \in \Go$, it is clear that we have
\begin{gather*}
AB1_{\Sscript{x}}=\underset{u \in \cH}{\oplus} \Bbbk \cG(x, \phio(u)).
\end{gather*}
Define left $A$-linear maps ${}^{*}e_{i}\colon AB1_{\Sscript{x}} \to A1_{\Sscript{u_{i}}}$ by sending $ b \in \cG(x,\phio(u)) \mapsto \Sf{E}_{\Sscript{(u_{i}, u)}}(bb_{i})$. Then, by hypothesis, the set $\{{}^{*}e_{i}, c_{i}\}_{\Sscript{i=1, \dots, N}}$ is a dual basis for the left $A$-module $AB1_{\Sscript{x}}$. Thus each of the modules $AB1_{\Sscript{x}}$ is finitely generated and projective, which is the first statement of part $(iii)$.

Now fixing $u \in \Ho$, we have a well defined left $B$-linear map
\begin{gather*}
\Psi_{u}\colon \ B1_{u} \longrightarrow B\hom{A-}{AB}{A1_{u}}, \qquad \big( b1_{u} \longmapsto \big[ ab' \mapsto \Sf{E}(ab'b) 1_{u} \big] \big),
\end{gather*}
which, by the naturality of $\Sf{E}$, is also natural. The maps $\Psi$ are bijective as they have inverses given by: $\mho\colon B\hom{A-}{AB}{A1_{u}} \to B1_{u}$ sending $\alpha \mapsto \sum_{i}b_{i}\alpha(c_{i})$. Indeed, for every $a \in A$ and $b' \in B$ with $ab' \in AB1_{u}$, we have
\begin{align*}
\Psi\mho(\alpha) (ab') &= \Psi\bigg( \sum_{i} b_{i}\alpha(c_{i})\bigg) (ab') \\
&= \sum_{i} \Sf{E}\big( ab'b_{i}\alpha(c_{i})\big) = \sum_{i} a \Sf{E}\big(b'b_{i}\big)\alpha(c_{i}) = \sum_{i} a \alpha\big( \Sf{E}\big( b'b_{i}\big)c_{i}\big) = \alpha(ab'),
\end{align*}
and so $\Psi\mho={\rm id}$. In the other way around, we have that $\mho \Psi(b1_{u}) = \sum_{i}b_{i}\Psi(b1_{u})(c_{i}) = \sum_{i} b_{i}\Sf{E}(c_{i}b) = b$. This gives the desired isomorphism and finishes the proof of part $(iii)$. \QEDA

\subsection[The proof of $(iii) \Rightarrow (i)$ in Theorem \ref{thm:FM}]{The proof of $\boldsymbol{(iii) \Rightarrow (i)}$ in Theorem \ref{thm:FM}}\label{sssec:III-IV}
To show this implication, it is sufficient to check, from one hand, that each of the functors ${\rm Nat}\big(\upphi^{\Sscript{x}}, \tilde{(-)}\big)\colon \Rep{\cH} \to \Vect$, with $x \in \Go$, preserves colimits, and from another one, there is a natural isomorphism $\Gamma_{\Sscript{(u, x)}}\colon \cG(\phio(u), x) \cong {\rm Nat}(\upphi^{\Sscript{x}}, \tilde{\cH_{u}})$, for ever pair $(u, x) \in \Ho \times \Go$.

The fact that ${\rm Nat}\big(\upphi^{\Sscript{x}}, \tilde{(-)}\big)$ preserves colimits, is deduced by using the first statement of $(iii)$ and the natural isomorphism
\begin{gather}\label{Eq:Calentaminetoglobal}
\hom{A-}{AB1_{x}}{M} \longrightarrow {\rm Nat}\big(\upphi^{\Sscript{x}}, \tilde{\oO(M)}\big),\qquad \big( f \longmapsto \big[ \upphi^{\Sscript{x}}_{\Sscript{v}} \to \tilde{\oO(M)}_{\Sscript{v}}, \big( b \mapsto f(b)
\big)\big]_{\Sscript{v \in \Ho}} \big),\!\!\!
\end{gather}
where $\oO\colon \lmod{A} \to \Rep{\cH}$ is the inverse functor of the following (symmetric monoidal) isomorphism of categories
\begin{gather*}
\oO^{-1}\colon \ \Rep{\cH} \longrightarrow \lmod{A}, \qquad \big( \cW \longrightarrow \oplus_{u \in \Ho} \Wu \big),
\end{gather*}
where $\lmod{A}$ denotes the category of unital left $A$-modules. In particular, this induces a natural isomorphism
\begin{gather*}
\hom{A-}{AB1_{x}}{A1_{u}} \cong {\rm Nat}\big(\upphi^{\Sscript{x}}, \tilde{\cH_{\Sscript{u}}}\big).
\end{gather*}

As for the second condition, if we assume that there is a left $B$-linear natural isomorphism $\Phi_{u}\colon B1_{u} \cong B\hom{A-}{AB}{A1_{u}}$, for every $u \in \Ho$, then we can consider $\Gamma_{\Sscript{(u, x)}}\colon \cG(\phio(u),x) \to {\rm Nat}\big(\upphi^{\Sscript{x}}, \tilde{\cH_{u}}\big)$ to be defined by
\begin{gather*}
\Gamma_{\Sscript{(u, x)}}\colon \ \cG(\phio(u), x) \longrightarrow {\rm Nat}\big(\upphi^{\Sscript{x}}, \tilde{\cH_{u}}\big), \\
\hphantom{\Gamma_{\Sscript{(u, x)}}\colon}{} \ \Big( b \longmapsto \big[ 1_{x}\Phi_{u}(b)\colon \upphi^{\Sscript{x}} _{\Sscript{v}} \to \Bbbk \cH(u,v), \;\big( b' \mapsto
\Phi_{u}(b)(b'1_{x})\big) \big]_{\Sscript{v \in \Ho} } \Big).
\end{gather*}
The fact that $\Gamma_{\Sscript{(u, x)}} \in {\rm Nat}\big(\upphi^{\Sscript{x}}, \tilde{\cH_{u}}\big)$, follows directly from the fact that $\Phi_{u}(b)$ is a left $A$-linear for every $b \in \cG(\phio(u),x)$. The naturality of $\Gamma$, that is, the commutativity of the diagrams of the form
\begin{gather*}
\xymatrix@C=45pt{ \cG(\phio(u), x) \ar@{->}[rr] \ar@{->}[d] & & {\rm Nat}\big(\upphi^{\Sscript{x}}, \tilde{\cH_{u}}\big) \ar@{->}[d] \\ \cG(\phio(u'), x') \ar@{->}[rr] & & {\rm Nat}\big(\upphi^{\Sscript{x'}},
\tilde{\cH_{u'}}\big) }
\end{gather*}
for every pairs $ a \in \cH(u',u)$ and $g\in \cG(x,x')$, is computed as follows. Take an element $b \in \cG(\phio(u), x)$, then, for every object $v \in \Ho$ and $ b' \in
\upphi^{\Sscript{x'}}_{\Sscript{v}}=\cG(x',\phio(v))$, we have
\begin{gather*}
{\rm Nat}\big(\upphi^{\Sscript{g}}, \tilde{\cH_{a}}\big) \circ \Gamma_{\Sscript{(u, x)}}(b)_{v}(b') = \Gamma_{\Sscript{(u, x)}}(b)_{v}(b'g) a = \Phi_{u}(b)(b'g) a = \Phi_{u'}(ba1_{u'})(b'g).
\end{gather*}
On the other hand, we have
\begin{gather*}
\Gamma_{\Sscript{(u', x')}}(gb\phi(a))_{v}(b') = \Phi_{u'}(gb\phi(a))(b') = \Phi_{u'}(b\phi(a))(b'g).
\end{gather*}
Comparing the two computations shows the commutativity of that diagram. Lastly, the inverse of $\Gamma_{\Sscript{(u, x)}}$ is provided by that of $1_{x}\Phi_{u}$ combined with the inverse of the natural isomorphism of equation~\eqref{Eq:Calentaminetoglobal}. This completes the proof of Theorem~\ref{thm:FM}.
\QEDA

\smallskip

The notion of a \emph{finite groupoid} is vague, in the sense that there are several interrelated notions and all generalize that of a finite group. Next we adopt the following one: a groupoid $\cG$ is said to be \emph{finite}, provided that $\pi_{0}(\cG)$ is a finite set as well as each of its isotropy groups $\cG^{\Sscript{x}}$ when $x$ runs in $\Go$. This is the case when for instance~$\Ga$ is a finite set. Given a morphism of groupoids $\upphi\colon \cH \to \cG$, we shall implicitly use the notations of Examples~\ref{exam: HG} and~\ref{exam:bisets}.

The following corollary characterizes Frobenius extension by subgroupoids, in terms of finiteness of the orbits of the fibres of the pull-back biset, which in particular applies to the case of
finite groupoids.

\begin{Corollary}\label{coro:Finite}Let $\upphi\colon \cH \to \cG$ be a morphism of groupoids such that $\upphi$ is a faithful functor and $\upphi_{\Sscript{0}}\colon \Ho \to \Go$ is an injective map $($e.g.,~$\phia\colon \Ha \to \Ga$ is an injective map$)$. Then the following are equivalent:
\begin{enumerate}[$(a)$]\itemsep=0pt
\item $\upphi$ is a Frobenius extension;
\item for any $x \in \Go$, the left $\cH$-set $\vartheta^{-1}\big( \{x\} \big)$ has finitely many orbits;
\item for any $x \in \Go$, the right $\cH$-set $\varsigma^{-1}\big( \{x\} \big)$ has finitely many orbits.
\end{enumerate}
In particular, any inclusion of finite groupoids is a Frobenius
extension.
\end{Corollary}
\begin{proof}Let us first check the equivalence between conditions $(c)$ and $(b)$. To this aim, we recall from Section~\ref{ssec:Grpd1}, that the category of right groupoid-sets is isomorphic to left groupoid-sets (both categories are defined over the same groupoid). In particular the image, under this isomorphism, of the right $\cH$-set $\varsigma^{-1}\big( \{x\} \big)$ is $\varsigma^{-1}\big( \{x\} \big)^{\Sscript{o}}$ its opposite left $\cH$-set. As a consequence, there is a~bijection between the orbits set of $\varsigma^{-1}\big( \{x\} \big)$ and that of $\varsigma^{-1}\big( \{x\} \big)^{\Sscript{o}}$, which sends any orbit to its opposite.
Thus, $\varsigma^{-1}\big( \{x\} \big)$ and has finitely many right $\cH$-orbits if and only if $\varsigma^{-1}\big( \{x\} \big)^{\Sscript{o}}$ has finitely many left $\cH$-orbits. Therefore, conditions $(c)$ and $(b)$ are equivalent, since we know that there is an isomorphism of left $\cH$-sets:
\begin{gather}\label{Eq:TS}
\vartheta^{-1}\big( \{x\} \big) \longrightarrow \varsigma^{-1}\big( \{x\} \big)^{\Sscript{o}}, \qquad \big( (u,c) \longmapsto \big(c^{-1},u\big)^{\Sscript{o}} \big).
\end{gather}

The implication $(a) \Rightarrow (b)$ is derived as follows from the first condition of Theorem \ref{thm:FM}$(iii)$. For any $x \in \Go$, we know that
\begin{gather*}
AB1_{x}=\bigoplus_{u \in \Ho} \Bbbk \cG(x, \phio(u))= \Bbbk \vartheta^{-1}\big( \{x\}\big).
\end{gather*}
Therefore, as left $A$-module, $AB1_{x}$ can be decomposed as direct sum of cyclic unital $A$-sub\-mo\-du\-les of the form
\begin{gather*}
AB1_{x} = \bigoplus_{(u, q) \in {\rm rep}_{\cH}(\vartheta^{-1}( \{x\}))} A q,
\end{gather*}
where ${\rm rep}_{\cH}\big(\vartheta^{-1}\big( \{x\}\big)\big)$ is a set of representative classes modulo the left $\cH$-action on the fibre $\vartheta^{-1}\big( \{x\}\big)$, and $A q$ is the $\Bbbk$-vector space spanned by the orbit set ${\rm Orb}_{\Sscript{\cH}}^{\Sscript{l}}(u,q)$. Since $AB1_{x}$ is finitely generated and projective, this direct sum should be then finite, which means that ${\rm rep}_{\cH}\big(\vartheta^{-1}\big(
\{x\}\big)\big)$ is a finite set and this is precisely condition~$(b)$.

As for the implication $(b) \Rightarrow (a)$, we assume that $\vartheta^{-1}\big( \{x\}\big)$ has finitely many orbits, for any $x \in \Go$. Choose a finite set of representatives classes
$\{(u_{\Sscript{i, x}}, q_{\Sscript{i, x}})\}_{\Sscript{i=1,\dots, n_{x}}}$ such that $\vartheta^{-1}\big( \{x\}\big)$ decomposes as
\begin{gather*}
\vartheta^{-1}\big( \{x\}\big)= \biguplus_{i=1}^{n}{\rm Orb}_{\Sscript{\cH}}^{\Sscript{l}}(u_{i, x},q_{i, x}),
\end{gather*}
a disjoint union of left $\cH$-subsets, see~\cite{Kaoutit/Spinosa:2018}. Using the left $\cH$-equivariant isomorphism of~\eqref{Eq:TS}, we also have
\begin{gather*}
\varsigma^{-1}\big( \{x\}\big) = \biguplus_{I=1}^{n} {\rm Orb}_{\Sscript{\cH}}^{\Sscript{l}}\big(q_{i, x}^{-1}, u_{i, x}\big).
\end{gather*}
In this way, we obtain as above a decomposition of unital $A$-modules
\begin{gather}\label{Eq:AB}
AB1_{\Sscript{x}}=\Bbbk \vartheta^{-1}\big( \{x\}\big) = \bigoplus_{i=1}^{n} A q_{i, x}\qquad \text{and}\qquad 1_{\Sscript{x}}BA= \Bbbk \varsigma^{-1}\big( \{x\}\big) =\bigoplus_{i=1}^{n} q_{i, x}^{-1} A.
\end{gather}
Let us denote by $\fk{p}_{\Sscript{i, x}}\colon AB1_{\Sscript{x}} \to Aq_{\Sscript{i, x}}$ and $\fk{p}_{\Sscript{i, x}}'\colon 1_{\Sscript{x}}BA \to q_{\Sscript{i, x}}^{-1}A$ the canonical projections attached to the direct sums of~\eqref{Eq:AB}. Take $x$ to be of the form $\phio(u)$, for some $u \in \Ho$. Then the element $(u, \iota_{\Sscript{\phio(u)}}) \in \vartheta^{-1}(\{\phio(u)\})$ and $(\iota_{\Sscript{\phio(u)}}, u) \in \varsigma^{-1}(\{\phio(u)\})$. Thus, if $\Orbit{u, \iota_{\Sscript{\phio(u)}}}{H}$ is equal to some orbit of the form $\Orbit{u_{\Sscript{i_0, \phio(u)}}, q_{\Sscript{i_0, \phio(u)}}}{H}$, for some $i_0$. Thus, we can assume that at least one of the $q_{\Sscript{i, \phio(u)}}$'s is $\iota_{\Sscript{\phio(u)}}$. Henceforth, we set $q_{\Sscript{1, \phio(u)}} :=\iota_{\Sscript{\phio(u)}}$. Therefore, we obtain the following $A$-linear maps
\begin{gather*}
\fk{p}_{\Sscript{1, \phio(u)}}\colon \ AB1_{\Sscript{\phio(u)}} \to A 1_{\Sscript{u}}\qquad \text{and} \qquad \fk{p}_{\Sscript{1, \phio(u)}}'\colon \ 1_{\Sscript{\phio(u)}}BA \to 1_{\Sscript{\phio(u)}}A.
\end{gather*}
Specifically, for a given arrow $p \in\cG(\phio(u), \phio(v))$ with $u,v \in \Ho$, we have that $(v,p) \in \vartheta^{-1}(\{\phio(u)\})$, and the image of the element $ (1_{\Sscript{\phio(v)}} p) 1_{\Sscript{\phio(u)}} \in AB 1_{\Sscript{\phio(u)}}$ by the projection $\fk{p}_{\Sscript{1, \phio(u)}}$ is given by the rule
\begin{gather*}
\fk{p}_{\Sscript{1, \phio(u)}} = \begin{cases} h 1_{\Sscript{\phio(u)}}, & (v,p) = h . (u,\iota_{\Sscript{\phio(u)}}), \\ 0, & \text{ if } (v,p) \notin \Orbit{u, \iota_{\Sscript{\phio(u)}}}{H}. \end{cases}
\end{gather*}
Since $\upphi$ is a faithful functor, this $h$ is uniquely determined from the pair $(v,p)$. Furthermore, for any arrow $\alpha \in \cH(u',u)$, we have that $\fk{p}_{\Sscript{1,\phio(u')}} (p \phia(\alpha)) = \fk{p}_{\Sscript{1, \phio(u)}}(p) \alpha$, as we know that ${}^{\Sscript{\upphi}}\uU(\cG):=\Ho \due \times {\Sscript{\upphi_0}}{ \Sscript{t}} \Ga$ is an $(\cH,\cG)$-biset. As a consequence, we obtain a natural transformation
\begin{gather*}
\Sf{E}_{\Sscript{(u, v)}}\colon \ \cG(\phio(u), \phio(v)) \longrightarrow \Bbbk\cH(u,v), \qquad \big( p \longmapsto \fk{p}_{\Sscript{1, \phio(u)}}(p) \big).
\end{gather*}

\looseness=-1 On the other hand, if we take an element $b' \in \cG(\phio(u), x)$, for some $x \in \Go$ and $u \in \Ho$, then we can write $b'= \lambda_{\Sscript{1}} q_{\Sscript{1, x}}^{-1} a_{\Sscript{1}} + \cdots + \lambda_{\Sscript{n}} q_{\Sscript{n, x}}^{-1} a_{\Sscript{n}}$, where all the scalars $\lambda_{\Sscript{i}}$'s vanish except the one which correspond exactly to the orbit that contains $b$ and its value is $1_{\Sscript{\Bbbk}}$. Now, each of the element $q_{\Sscript{j, x}} b'$ has the image $\Sf{E}(q_{\Sscript{j, x}} b') = \lambda_{\Sscript{j}} a_{\Sscript{j}}$. Therefore, we have that $b' =\sum_j q_{\Sscript{j, x}}^{-1} \Sf{E}(q_{\Sscript{j, x}} b')$ as en element in the homogeneous component $\Bbbk \cG(\phi(u), x)$ of $B$. If we take now an element $b \in \cG(x,\phi(u))$, for some $x \in \Ga$ and $u \in \Ho$, then the same arguments will show that $ b = \sum_j \Sf{E}(bq_{\Sscript{j, x}}^{-1}) q_{\Sscript{j, x}}$ as an element in the component $\Bbbk \cG(x,\phi(u))$. In summary, we have shown condition~$(ii)$ of Theorem~\ref{thm:FM}, and thus $\upphi$ is a~Frobenius extension. The particular statement is now immediate, and this finishes the proof.
\end{proof}

The subsequent corollary is a direct consequence of Corollary \ref{coro:Finite} and Theorem \ref{thm:FM}, some of the implications stated there, can be deduced form \cite[Example~1.7] {Kadison:1999} and \cite[pp.~36, 39]{Kowalski:2014}. For the definition of Frobenius extension of algebras, we refer to \cite[Definition 1.1] {Kadison:1999}.

\begin{Corollary}\label{cor:FiniteI}Let $\upphi\colon H \to G$ be a monomorphism of groups, and consider the associated group algebra extension $\phi\colon A=\Bbbk H \to B=\Bbbk G$. Then the following are equivalent:
\begin{enumerate}[$1)$]\itemsep=0pt
\item $\upphi$ is a Frobenius extension;
\item the image of $H$ in $G$ is a finite-index subgroup;
\item $B/A$ is a Frobenius extension of algebras.
\end{enumerate}
In particular, all conditions hold true for any inclusion of finite groups.
\end{Corollary}

\subsection*{Acknowledgements}
Research supported by the Spanish Ministerio de Econom\'{\i}a y Competitividad and the European Union FEDER, grant MTM2016-77033-P.
The authors would like to thanks the referees for their comments and suggestions, which help us to improve the earlier version of this paper. We are also grateful to Paolo Saracco for his careful reading and for his comments.

\pdfbookmark[1]{References}{ref}
\LastPageEnding

\end{document}